\documentclass[11pt]{amsart}
\usepackage[bottom=1in, right=1in, left=1in, top=1in]{geometry}

\RequirePackage[OT1]{fontenc}
\RequirePackage{amsthm,amsmath}
\RequirePackage{comment}
\RequirePackage{graphicx,subfigure,latexsym,amssymb}
\RequirePackage{float,epsfig,multirow,rotating,times}
\RequirePackage{upgreek,wrapfig}
\usepackage{amsthm, amssymb}
\usepackage{tikz}
\usetikzlibrary{matrix}
\usepackage{nicematrix}
\usepackage{arydshln}
\usepackage{blkarray}
\usepackage{amsthm, amssymb}
\newtheorem{theorem}{Theorem}[section]
\newtheorem{proposition}[theorem]{Proposition}
\newtheorem{corollary}[theorem]{Corollary}
\newtheorem{lemma}[theorem]{Lemma}
\theoremstyle{definition}
\newtheorem{example}[theorem]{Example}
\theoremstyle{definition}
\newtheorem{definition}[theorem]{Definition}
\theoremstyle{definition}

\usepackage{mathtools} 

\newcommand{\rows}{\mathrm{rows}}
\newcommand{\cols}{\mathrm{cols}}
\renewcommand{\int}{\mathrm{Int}}
\newcommand{\Ibb}{\mathbb{I}}
\renewcommand{\max}{\mathrm{Max}}
\newcommand{\level}{\mathrm{Level}}

\newcommand{\Acal}{\mathcal{A}}
\newcommand{\Abar}{\bar{\mathcal{A}}}
\newcommand{\Bcal}{\mathcal{B}}
\newcommand{\Gcal}{\mathcal{G}}
\newcommand{\Mcal}{\mathcal{M}}
\newcommand{\bfy}{\mathbf{y}}

\usepackage[colorlinks,citecolor=blue,urlcolor=blue]{hyperref}

\title[Geometry of rational quasi-independence models as toric fiber products]{Geometry of rational quasi-independence models \\ as toric fiber products}

\author{Jane Ivy Coons}
\address{St John's College and Mathematical Institute, University of Oxford}
\email{jane.coons@maths.ox.ac.uk}

\author{Heather A. Harrington}
\address{Mathematical Institute, University of Oxford and}
\address{Wellcome Centre for Human Genetics, University of Oxford and}
\address{Max Planck Institute of Molecular Cell Biology and Genetics, Dresden and}
\address{Centre for Systems Biology Dresden, and}
\address{Faculty of Mathematics, Technische Universit\"at Dresden}
\email{harrington@mpi-cbg.de}

\author{Niharika Chakrabarty Paul}
\address{Max Planck Institute for Mathematics in the Sciences, Leipzig}
\email{niharika.paul@mis.mpg.de}

\begin{document}

\begin{abstract}
    We investigate the geometry of a family of log-linear statistical models called quasi-independence models. The toric fiber product is useful for understanding the geometry of parameter inference in these models because the maximum likelihood degree is multiplicative under the TFP. We define the coordinate toric fiber product, or cTFP, and give necessary and sufficient conditions under which a quasi-independence model is a cTFP of lower-order models. We show that the vanishing ideal of every 2-way quasi-independence model with ML-degree 1 can be realized as an iterated toric fiber product of linear ideals. We also classify which Lawrence lifts of 2-way quasi-independence models are cTFPs and give a necessary condition under which a $k$-way model has ML-degree 1 using its facial submodels.
\end{abstract}

\maketitle

\section{Introduction}

Log-linear models are discrete statistical models that are parameterized by monomials. This large family of models include well-known graphical and hierarchical models. Tools from combinatorics, algebraic geometry and commutative algebra can help to understand statistical properties of the model.
In the algebraic statistics setting, the Zariski closure of a log-linear model is a toric variety and one can use toric geometry to study inference problems.
The log-linear models that we focus on here are called $k$\emph{-way quasi-independence models.} A collection of $k$ discrete random variables that each have a finite number of states satisfy \emph{quasi-independence} if there are some combinations of states that cannot co-occur, but they are otherwise independent. A $k$-way quasi-independence model is therefore specified by a set of tuples $S$, in the product of the state spaces of the $k$ discrete variables, that indicates the states that can occur together. These models have been investigated widely in the statistical theory literature and in applications \cite{bocci2017exact, colombo1988quasi,goodman1994quasi}. 

Given some data $u$ and a parametric statistical model $\mathcal{M}$, the maximum likelihood estimator (MLE) for $u$ in $\mathcal{M}$ is the distribution in $\mathcal{M}$ that maximizes the probability of observing $u$. The MLE can be computed by optimizing the \emph{log-likelihood function}, whose derivatives are rational functions in the parameters of $\mathcal{M}$ in the log-linear case. The number of complex critical points of the log-likelihood function for generic data is called the \emph{ML-degree} of the model. Much recent work in algebraic statistics has explored ML-degrees of discrete models, such as \cite{amendola2019maximum,catanese2006maximum,coons2021quasi, huh2013maximum,huh2014likelihood}. 
 In the present work, we are especially interested in the toric geometry of quasi-independence models with ML-degree 1. Discrete models with ML-degree 1 have been characterized using the Horn uniformization and the theory of A-discrimants in \cite{duarte2021discrete, huh2014likelihood}; however, we still lack a framework for practically determining whether a log-linear model has rational MLE based on its parametrization. 
 The $2$-way quasi-independence models with rational MLE were characterized using solely the combinatorial features of the $\Acal$-matrix in \cite{coons2021quasi}.
 A similar such classification remains an open question for $k > 2$. Here we consider quasi-independence models with rational MLE that are obtained from lower-order models using the toric fiber product operation on their $\Acal$-matrices.
 
The toric fiber product (TFP) is a product of two ideals according to a common and specified multigrading \cite{sullivant2007toric}. In algebraic statistics, the TFP is a commonly used operation that allows for the iterative generation of statistical models from those of lower dimension. Viewing log-linear models as toric fiber products is especially convenient when working in the paradigm of maximum likelihood estimation, as the ML-degree is multiplicative under the toric fiber product \cite{amendola2020maximum}. In the case where the vanishing ideal of the model is the TFP of several linear ideals, the iterative proportional scaling algorithm for computing the MLE exhibits exceptional behavior: it produces the \emph{exact} MLE for the given data in a single step \cite{coons2022rational}.

 In this paper, we define a special type of TFP based on the integer matrix of the monomial parametrization of a log-linear model. We call this a \emph{coordinate toric fiber product}, or cTFP. In Section 3, we give two necessary and sufficient conditions based on this parametrizing matrix under which a $k$-way quasi-independence model is a toric fiber product of two lower-order models. We further develop the combinatorial theory of $2$-way quasi-independence models with rational MLE in Section 4. We use this and our results on cTFPs to prove the following main results.
 
 \begin{theorem} Let $\Mcal$ be a $2$-way quasi-independence model with rational MLE and let $I$ be its toric vanishing ideal.
 \begin{enumerate}
      \item (Theorem \ref{2waythm}) The model $\Mcal$ has an $\Acal$-matrix that realizes $\Mcal$ as a coordinate toric fiber product along each of its internal coordinates.
      \item (Theorem \ref{thm:LinearIdeals}) The ideal $I$ can be realized as an iterated toric fiber product of linear ideals.
\end{enumerate}
 \end{theorem}
 In light of the results of \cite{coons2022rational}, this surprising finding also implies that every such model has a parametrization under which the iterative proportional scaling algorithm exhibits one-cycle convergence. We next consider the richer structure of quasi-independence models by exploring connections to hypergraphs and tensors.

We consider \emph{Lawrence lifts}: an operation that is performed on a graph that is associated to a $2$-way quasi-independence model to generate a hypergraph that is associated to a larger $k$-way model. We provide a necessary and sufficient condition, on the graphs of the $2$-way models, for the $\Acal$-matrix associated to the Lawrence lift to be a coordinate toric fiber product. We combine our results with those of Brysiewicz and Maraj \cite{brysiewicz2023lawrence} to construct a family of quasi-independence models that are ML-degree one but have $\Acal$-matrices that are not toric fiber products. 

We can also view a quasi-independence model to be associated to a tensor and to a polytope, with faces that correspond to the different `slices' of the tensor. Each of these faces are associated to a model, called a \emph{facial submodel}. We prove that if this model has ML-degree $1$, then each of the facial submodels must also have ML-degree $1$. We also provide a counterexample to the converse. 

 We structure this paper as follows. In Section 2, we provide preliminaries on quasi-independence models, maximum likelihood estimation and toric fiber products. We then present the two necessary and sufficient conditions for a $k$-way quasi-independence model to be a cTFP in Section 3. In Section 4, we give a reparametrization of any 2-way quasi-independence model with rational MLE. We provide the necessary and sufficient condition for the model associated to a Lawrence lift to be a TFP in Section 5. In Section 6, we recall some theory of facial submodels and submatrices and use this to give conditions under which the model has ML-degree greater than one.

 \section{Preliminaries}

\subsection{Quasi-independence models }

We begin by reviewing the definitions of parametric and log-linear statistical models. We then introduce quasi-independence models, which are the types of log-linear model that we focus on in the present work.

\begin{definition}A \textbf{statistical model} $\mathcal{M}$ is a collection of probability distributions or density functions. A \textbf{parametric statistical model} $\mathcal{M}_\Theta$ is the image of a map $p$ from finite-dimensional $\Theta \subseteq \mathbb{R}^d$ to a space of probability density or distribution functions. 
Denoting by $p_{\theta}$ the image of $\theta \in \Theta$ under $p$, the parametric statistical model is $\mathcal{M}_\Theta = \{p_\theta: \theta \in \Theta\}$. We refer to $\Theta$ as the {\bf parameter space} of the model.

Denote by $\Delta_{r-1}$ the $(r-1)$-dimensional probability simplex and let $\mathcal{A} \in \mathbb{Z}^{d\times r}$ be a matrix with entries $a_{ij}$.  We assume that the all ones vector $\bf{1}$ lies in the rowspan of $\mathcal{A}$. Then, the \textbf{log-linear model} associated to the given $\mathcal{A}$-matrix is the set of probability distributions: $\mathcal{M}_{\mathcal{A}} = \{p \in \Delta_{r-1} | \log(p) \in \text{rowspan}(A)\}.$
\end{definition}

Equivalently, we may define the map $\phi^{\mathcal{A}}: \mathbb{R}^d \rightarrow \mathbb{R}^r$ coordinate-wise by
$\phi^{\mathcal{A}}_j (t_1 \ldots t_d) = \prod_{i=1}^d t_i^{a_{ij}}$.
Then $\mathcal{M}_{\mathcal{A}} = \phi^{\mathcal{A}}(\mathbb{R}^d) \cap \Delta_{r-1}$.
From this definition, we see that the complex Zariski closure of a log-linear model is a toric variety as it is the image of a monomial map \cite{sullivant2018algebraic}. Its vanishing ideal $I_{\Acal} := I(\mathcal{M}_{\Acal})$ is a toric ideal whose binomials are determined by the integer kernel of $\mathcal{A}$. In particular, a binomial $x^u - x^v$ belongs to $I_{\mathcal{A}}$ if and only if $u-v \in \ker_{\mathbb{Z}}({\mathcal{A}})$ \cite{sturmfels1996grobner}.

Let $X_1,\dots,X_k$ be discrete random variables on state spaces $[n_1],\dots,[n_k]$ respectively. These random variables satisfy \emph{quasi-independence} if there are some combinations of states that cannot occur together, but they are otherwise mutually independent. Let $S \subset [m_1] \times \dots \times [m_k]$. This set of tuples represents the states that can occur together. In order to parametrize this model, we index the coordinates of $\mathbb{R}^{m_1,\dots,m_k}$ by $(s^1_1,\dots,s^1_{m_1},s^2_1,\dots,s^k_1,\dots,s^k_{m_k})$.

\begin{definition} 
    Define the monomial map $\phi^S : \mathbb{R}^{m_1 + \ldots +m_k} \longrightarrow \mathbb{R}^S$ by
    $\phi^{S}_{i_1\ldots i_k}(s^1, \ldots, s^k) = \prod_{j=1}^ks^j_{i_j}.$
    The \textbf{$k$-way quasi-independence model} associated to $S$ is the model:
    $$\mathcal{M}_S := \phi^S(\mathbb{R}^{m_1 + \ldots +m_k}) \cap \Delta_{\# S-1}.$$
    Denote by $\mathcal{A}_S$ the $\mathcal{A}$-matrix of this monomial map. Its columns are indexed by $S$ and its rows split into $k$ blocks according to the $k$ coordinates of the tuples in $S$.
\end{definition}

We now consider 2-way quasi-independence models on random variables $X$ and $Y$ with state spaces $[m]$ and $[n]$, respectively.
To any 2-way quasi-independence model specified by $S$, we can associate a bipartite graph $\mathcal{G}_S$ with. The vertex set of $\mathcal{G}_S$ can be split into $[m]$ and $[n]$ and an edge $(i,j)$ belongs to the graph if and only if $(i,j) \in S$.

\begin{example}\label{one}
Consider the following set of pairs $S = \{(1,1), (1,2), (1,3), (2,1), (2,2), (3,1)\}$. $S$ is then associated to a $2$-way quasi-independence model. We can also represent $S$ as the following ``star matrix", which has stars in the entries corresponding to elements of $S$ and zeros elsewhere:
$$\begin{bmatrix}
\star & \star & \star\\
\star & \star & 0 \\
\star & 0 &0 \\
\end{bmatrix}$$

The $\mathcal{A}$-matrix of this model is a block matrix where each column corresponds to one of the coordinates that comprise the set $S$. Below is the associated $\mathcal{A}$-matrix $\mathcal{A}_S$ and bipartite graph $\mathcal{G}_S$. In the $\mathcal{A}$-matrix, the rows labelled $a_i$ represent the states of the first variable and those labelled $b_i$ represent the states of the second. In the graph, the rectangular vertices are associated to states of the first variable and the circular vertices are associated to states of the second. We then form an edge between the rectangle that represents $a_i$ and the circle that represents $b_j$ if $ij$ is in $S$.

\begin{figure}[ht!]
    \centering
    \begin{minipage}{0.49\textwidth}
 \[\begin{blockarray}{ccccccc}
    &11 &12 &13 &21 &22 &31 \\
\begin{block}{c(cccccc)}
a_1 & 1 & 1 & 1 & 0 & 0 & 0\\
a_2 & 0 & 0 & 0 & 1 & 1 & 0 \\
a_3 & 0 & 0 & 0 & 0 & 0 & 1 \\
    \cline{2-7}
b_1 & 1 & 0 & 0 & 1 & 0 & 1 \\
b_2 & 0 & 1 & 0 & 0 & 1 & 0 \\
b_3 & 0 & 0 & 1 & 0 & 0 & 0 \\
\end{block}
 \end{blockarray}\]
    \end{minipage}
    \begin{minipage}{0.49\textwidth}
    \centering
    \begin{tikzpicture}[every node/.style={minimum size=0.5cm}]
\node [shape=rectangle, draw=black] (A1) at (0,1) {1};
\node [shape=circle, draw=black] (B1) at (0,0) {1};
\node [shape=circle, draw=black] (B2) at (1,1) {2};
\node [shape=rectangle, draw=black] (A2) at (1,0) {2};
\node [shape=rectangle, draw=black] (A3) at (2,1) {3};
\node [shape=circle, draw=black] (B3) at (2,0) {3};

\path [-] (A1) edge (B1);
\path [-] (A1) edge (B2);
\path [-] (A1) edge (B3);
\path [-] (A2) edge (B1);
\path [-] (A2) edge (B2);
\path [-] (A3) edge (B1);
\end{tikzpicture}
        \end{minipage}
           
\end{figure}

\end{example}

We can similarly represent a $k$-way quasi-independence model, via a $k$-dimensional star-tensor or an $\mathcal{A}$-matrix with $k$ blocks. A $k$-way model would also have an associated $k$-partite hypergraph. The $\mathcal{A}$-matrices of $k$-way quasi-independence models were called \emph{multipartition matrices} by Coons, Langer, and Ruddy \cite{coons2022rational}. We define these and consider more examples of $\mathcal{A}$-matrices of quasi-independence models in Section \ref{sec:TFP}

\subsection{Maximum Likelihood Estimators and Maxmimum Likelihood Degrees}

Given a parametric model and some empirical data, we wish to determine the distribution in the model that best fits the data. One way to achieve this is through computing the maximum likelihood estimate, which is the distribution in the model that maximizes the probability of observing the given data.

\begin{definition}
   Given data $D$ from a discrete statistical model $\mathcal{M}$, the \textbf{likelihood function} is
$L(p | D) := p(D)$
where $p(D)$ is the probability of observing the data under the distribution $p \in \mathcal{M}$.
The \textbf{maximum likelihood estimate} (MLE) $\hat{p}$, if it exists, is the distribution  that maximizes the likelihood over the model:
$\hat{p} = {\text{arg max}}_{p \in \mathcal{M}} L(p | D)$. The \textbf{log-likelihood function} $l(p|D)$ is the natural logarithm of the likelihood function.
\end{definition}

To analyze the MLE problem we study the log-likelihood function $l(p|D)$. Since the logarithm is concave, we know that the log-likelihood function and likelihood are maximized at the same values of $p$. In the case of log-linear models, the derivatives of the log-likelihood function are rational functions, which facilitates an algebraic approach.

In the case of independent, identically distributed (iid) data in a discrete statistical model with finite support, the likelihood function has a simple form. Let $\mathcal{M}$ be such a model whose distributions are of the form $p = (p_s)_{s \in S}$ for some finite index set $S$. Let $\mathbf{u} \in \mathbb{N}^{\#S}$ be an iid vector of counts. Then, the likelihood function in this case is
$L(p|u) = \prod_{s \in S} p_s^{u_s}$
and the log-likelihood function is
\[
l(p|u) = \sum_{s \in S} u_s\log(p_s).
\]

\begin{definition}
    The \textbf{maximum likelihood degree} (ML-degree) of a model is the number of complex critical points of the log-likelihood function counted with multiplicities for generic data. In the case where the model has ML-degree 1, the MLE can be written as a rational function of the data \cite[Chapter~7.2]{sullivant2018algebraic}. We say that such a model has \textbf{rational MLE}.
\end{definition}

The following theorem, Birch's Theorem, gives a form of the MLE, if it exists. A proof can be found in \cite{sullivant2018algebraic}, or \cite{lauritzen1996graphical}, for example. 
\begin{theorem}[Birch's Theorem]\label{birchthm}  Let $\mathcal{A} \in \mathbb{Z}^{n\times r}$ such that $\mathbf{1} \in \text{rowspan} ({\mathcal{A}})$. Let $u \in \mathbb{R}^r_{\geq 0}$ such that $u_+ = u_1 + \ldots u_r$. Then, if the maximum likelihood estimate of the model $\mathcal{M}_{\mathcal{A}}$ exists, it is the unique solution to the system of equations: $\mathcal{A} u =$$ u_+ \mathcal{A}p $ and $p \in \mathcal{M}_{\mathcal{A}}$.
\end{theorem}
In the context of algebraic geometry, Theorem \ref{birchthm} tells us that the MLE is the unique intersection point of the affine linear space, defined by $\mathcal{A}u = u_+ \mathcal{A}p$, and the positive part of the toric variety $\overline{\mathcal{M}_{\mathcal{A}}}.$
Coons and Sullivant \cite{coons2021quasi} used this algebraic perspective to give a characterization of the 2-way quasi-independence models with rational MLE using the bipartite graph $\mathcal{G}_S$.

\begin{definition}
    A bipartite graph is \textbf{doubly chordal} if each cycle of length greater than or equal to 6 has two chords. Equivalently, it is doubly chordal if it has no induced subgraph that is a cycle of length greater than or equal to 6, or the ``double square" graph obtained by gluing two 4-cycles along a common edge.
\end{definition}

\begin{theorem} \cite{coons2021quasi} \label{janethm}
    Let $S \subset [m] \times [n]$ be a set of indices with associated bipartite graph $\mathcal G_S$ and quasi-independence model $\mathcal M_S$. Then, $\mathcal M_S$ has ML-degree one if and only if $\mathcal G_S$ is doubly chordal bipartite. 
\end{theorem}

The \emph{toric fiber product} (TFP) is an operation on log-linear models under which ML-degrees are multiplicative. They serve as a way to create larger quasi-independence models from smaller building blocks. We introduce these operations in the next section, in order to eventually show that all $2$-way quasi-independence models are toric fiber products of models that are the intersection of the probability simplex with a linear space.

\subsection{Toric fiber Products}\label{sec:TFP}

The toric fiber product operation was defined by Sullivant in \cite{sullivant2007toric} as an operation on polynomial ideals that are homogeneous with respect to the same multigrading. In this section, we define this operation and describe its interpretation in the context of quasi-independence models.

\begin{definition}
For any polynomial ring $\mathbb{C}[t] : = \mathbb{C}[t_1,\dots,t_\ell]$, we can impose a \textbf{multigrading}. A multigrading is an assignment of a multidegree vector to each monomial $t^u$ in the ring, such that $\text{multideg}(t^{u+v}) = \text{multideg}(t^u) + \text{multideg}(t^v)$. 
A polynomial is homogeneous with respect to the multigrading if each of its nonzero terms have the same multidegree. Then, an ideal $I \in \mathbb{C}[t]$ is \textbf{homogeneous with respect to the multigrading} if it has a generating set consisting of homogeneous polynomials. 
\end{definition}

Let $r$ be a fixed positive integer. For each $i \in [r]$, fix positive integers $s_i$ and $t_i$. Then, define the following rings:
\begin{align*}
    \mathbb{C}[x]& = \mathbb{C}[x^i_j | i \in [r], j \in [s_i]]\\
    \mathbb{C}[y]& = \mathbb{C}[y^i_k | i \in [r], k \in [t_i]]\\
    \mathbb{C}[z]& = \mathbb{C}[x^i_{jk} | i \in [r], j \in [s_i], k \in [t_j]].
\end{align*}

Fix multigradings on $\mathbb{C}[x]$ and $\mathbb{C}[y]$ such that $\text{multideg}(x^i_j) =\text{multideg}(y^i_j) = \mathbf{d}^i$. Define $D$ to be the matrix with columns $\mathbf{d}^i$. We will refer to $D$ as the multigrading matrix. We assume throughout that the columns of $D$ are linearly independent; in the literature, these are sometimes referred to as \emph{codimension zero TFPs} and one can remove this assumption to obtain \emph{higher codimension TFPs} \cite{RAUH2016276}.
If $I \subset \mathbb{C}[x]$ and $J \subset \mathbb{C}[y]$ are ideals that are homogeneous with respect to the multigrading, we can define the ring homomorphism:
\begin{eqnarray}
\phi_{I,J}: \mathbb{C}[z] \rightarrow& (\mathbb{C}[x]/I) \otimes_\mathbb{C} (\mathbb{C}[x]/I) \nonumber \\
z_{jk}^i \mapsto& x_j^i \otimes_\mathbb{C} y_k^i. \nonumber
\end{eqnarray}

\begin{definition}\label{tfpdef} The \textbf{toric fiber product} (TFP) of $I$ and $J$ with respect to $D$ is: $I \times_D J := ker(\phi_{I, J})$. We also say that the variety $V(I \times_D J)$ is the toric fiber product of the varieties $V(I)$ and $V(J)$.
\end{definition}

Here we are specifically interested in the TFPs of quasi-independence models and how to build a given quasi-independence model using TFPs. In the case of log-linear models, the toric fiber product can be constructed on the level of the $\mathcal{A}$-matrices by simply ``gluing" columns together according to their multigrading. 

\begin{definition} \cite{coons2022rational}
    A $0/1$ matrix $A \in \mathbb{R}^{n \times m}$ is a \textbf{multipartition matrix} if the rows of $A$ can be partitioned into submatrices $A^1, \ldots A^p$, where in each $A^i, i \in \{1, \ldots, p\}$ the sum of the column entries sum to $1$. Then, the $A^i$ are referred to as the \textbf{partition matrices} or \textbf{blocks} of $A$. 
\end{definition}

We note that the $\mathcal{A}$-matrix of a quasi-independence model is always a multipartition matrix as the parameters of the model are grouped by their corresponding random variable. However, multipartition matrices are also allowed to have repeated columns.
Let us consider two $\mathcal{A}$-matrices of two quasi-independence models of any order, specified by tuples $S_1$ and $S_2$. We write
\[
\Acal_{S_1} = \begin{pmatrix}
\alpha \\
\hline 
\beta \end{pmatrix} \qquad \text{and} \qquad \Acal_{S_2} = \begin{pmatrix}
\overline{\beta} \\
\hline
\gamma \end{pmatrix}.
\]

Here $\beta$ is one partition matrix of $\mathcal{A}_{S_1}$ and $\alpha$ represents the other $k-1$ blocks of the matrix. Similarly, $\overline\beta$ is one block of $\mathcal{A}_{S_2}$ and $\gamma$ represents the other blocks of the matrix. We further suppose that $\beta$ and $\overline\beta$ have the same number of rows. Let us index a column of $\mathcal{A}_{S_1}$ by $\bf{i}j$, where $\bf{i}$ is a list of coordinates corresponding to blocks $\alpha$ and $j$ is associated to $\beta$. We denote this column of $\mathcal{A}_{S_1}$ by $\alpha_{\mathbf{i}} \beta_j$ and the restriction of this column to the rows of $\beta$ by $\beta_j$. Similarly, we index a column of $\mathcal{A}_{S_2}$ by $j\mathbf{k}$. 

In this case, we can assign the column indexed by $\bf{i}j$ the multidegree vector $\beta_j$ and the column indexed by $j\bf{k}$, $\overline\beta_j$. Then, the multigrading matrix is formed of the columns of $\beta$, or equivalently $\overline\beta$, without repetitions. Hence $D$, in this case, is the identity matrix with as many rows as $\beta$, or equivalently $\overline\beta$.

Let $I_1$ be the vanishing ideal of the log-linear model specified by $\mathcal{A}_{S_1}$ and $I_2$ be the ideal of the model associated to $\mathcal{A}_{S_2}$.
The rows of $\beta$ are by definition in the rowspan of  $\mathcal{A}_{S_1}$ and the rows of $\overline\beta$ are in the rowspan of $\mathcal{A}_{S_2}$.  Hence the two ideals $I_1$ and $I_2$ are homogeneous with respect to the multigrading and we may take their toric fiber product with respect to $D$.

The $\mathcal{A}$-matrix for this TFP is constructed as follows. For each pair of columns with the same multidegree, $\alpha_{\mathbf{i}}\beta_j$ and $\overline\beta_j\gamma_{\mathbf{k}}$ such that $\beta_j = \overline\beta_j$, we obtain a column in the matrix representing the toric fiber product of the two matrices with respect to this chosen multigrading. The column is of the form will have the following form: $(\alpha_i^T \ \ | \ \ \beta_j^T = \overline\beta_j^T \ \ | \ \ \gamma_k^T)^T$. The matrix representing the toric fiber product consists exactly of the columns obtained in this way. On the level of sets of tuples $S_1$ and $S_2$, each $\bf{i}$$j$ and $\bf{k}$$j$ glue together to form $\bf{ik}$$j$. The toric fiber product of the two models is then the quasi-independence model specified by all tuples of this form.

\begin{definition} A toric fiber product that is formed by the above procedure is called a \textbf{coordinate toric fiber product} (cTFP). 
\end{definition}

The multidegrees of a coordinate toric fiber product can equivalently be specified on the level of the polynomial ideals, columns of the corresponding $\mathcal{A}$-matrices or tuples in the index set $S$. We therefore slightly abuse terminology and refer to a TFP of $\mathcal{A}$-matrices or sets of tuples.

\begin{example}\label{eg1}
Consider the  $2$-way quasi-independence models specified by the sets of tuples in $[3] \times [3]$, 
\begin{align*}
    S_1 &= \{(1,1), (1,3), (2,1), (2,2), (3,3)\}\\
    S_2 &= \{(1,1), (1,3), (2,1), (3,2), (3,3)\}.
\end{align*}
Their corresponding $\mathcal{A}$-matrices are:
\[\mathcal{A}_{S_1}=\begin{pmatrix}
1 & 1 & 0 & 0 & 0 \\
0 & 0 & 1 & 1 & 0  \\
0 & 0 & 0 & 0 & 1  \\
\hline 
1 & 0 & 1 & 0 & 0  \\
0 & 0 & 0 & 1 & 0  \\
0 & 1 & 0 & 0 & 1 \\
\end{pmatrix} \qquad
\text{and} \qquad \mathcal{A}_{S_2}=
\begin{pmatrix}
1 & 1 & 0 & 0 & 0 \\
0 & 0 & 1 & 0 & 0  \\
0 & 0 & 0 & 1 & 1  \\
\hline 
1 & 0 & 1 & 0 & 0  \\
0 & 0 & 0 & 1 & 0  \\
0 & 1 & 0 & 0 & 1 \\
\end{pmatrix}.\]

Every pair of $(i,j) \in S_1$ is assigned the multidegree $e_j \in \mathbb{R}^3$. Similarly,
every pair in $(j,k) \in S_2$ is assigned the multidegree $e_j$. So, when we take the toric fiber product, we merge together each $(i,j)$ and $(j,k)$ to produce triples of the form $(i,j,k)$

So, the toric fiber product of these two models is the quasi-independence model specified by the set of triples,
$$ S = \{(1,1,1), (1,1,3), (1,3,2), (1,3,3), (2,1,1), (2,1,3), (2,2,1), (3,3,2), (3,3,3)\}.$$
\end{example}

As mentioned above, uncovering an underlying TFP structure is especially helpful when considering maximum likelihood estimation for log-linear models. This arises from the following multiplicative properties. 

 \begin{theorem}\label{amenmle}\cite{amendola2020maximum}
 Let $B$ and $C$ be the matrices representing the log-linear models and let $I$ and $J$ be the corresponding vanishing ideals of these models. Let $D$ be a suitable multigrading matrix and $\hat{p}(B)$, $\hat{p}(C)$, and $\hat{p}(D)$ be the MLEs of the associated log-linear models. Then, if D has linearly independent columns, the $(i, j, k)$th coordinate function of the maximum likelihood estimator for the log-linear model associated to the toric fiber product $I \times_D J$ is
$$\frac{\hat{p}^i_j(B)\hat{p}^i_k(C)}{\hat{p}^i(D)}.$$
 \end{theorem}

\begin{corollary}\cite{amendola2020maximum}
With the models and ideals defined as in the previous theorem, the ML-degree of the log-linear model associated to $I \times_D J$ is the product of the ML-degrees of the models associated to $I$ and $J$.
\end{corollary}

\section{Quasi-Independence Models as Toric fiber Products}
In this section, we provide two equivalent necessary and sufficient conditions under which the $\mathcal{A}$-matrix of a quasi-independence model is a coordinate toric fiber product of two lower-order quasi-independence model. We first introduce some new notation that we will use throughout this section.

\begin{definition}
    Let $S$ be a set of $k$-tuples that specify the $k$-way quasi-independence model $\mathcal{M}_S$. Let $s$ be any k-tuple in $S$ of the form $(i_1, \ldots i_{k})$. Fix $j \in [k]$ and let $a \subset [k]$ with $j \in a$. 
    Let $A := (i_\ell)_{\ell \in a}$ and $B := (i_\ell)_{\ell \in ([k] \setminus a) \cup \{j\}}$ In other words, $A$ is the portion of the tuple $s$ indexed by elements of $a$ and $B$ is the complementary portion of the tuple including the $j$th coordinate as well.
    We then say that $s$ is a \textbf{$j$-coordinate concatenation} of $A$ and $B$, according to $a$ using the notation $s = A +^{a}_j B.$ We say that $A$ and $B$ form a \textbf{$j$-coordinate split} of $s$ according to $a$.

    We may equivalently record the indices of the entries of $A$ as $\mathbf{in}(A)$ and write $S = A +_j^{\mathrm{in}(A)} B$. Note that we do not need to include the data of $\mathrm{in}(B)$ since it is entirely determined by $\mathrm{in}(A)$ and $j$. When we have fixed $\mathrm{in}(A)$, we use the notation $s = A +_j B$ for simplicity.
\end{definition}


We define these $j$-coordinate splits as they give us candidates for coordinate toric fiber factors. For example, if we have the tuple $(v,w,x,y,z)$, we can choose $A = (v,x,y)$ and $B = (w,x,z)$. Then, $\mathrm{in}(A)$ is $(1,3,4)$ and $\mathrm{in}(B)$ is $(2,3,5)$ and $(v,w,x,y,z) = A +_3 B$. 

We denote by $\#_S(x)$, the number of times an element $x$ occurs in the multiset $S$; we call this the \emph{frequency} of $x$ in $S$.
We are now ready to state the first necessary and sufficient condition below. This is a combinatorial characterization of the coordinate partitions of $S$ that realize it as a TFP of two lower-order models.

\begin{proposition}\label{mainlemma}
    Let $S$ be a set of k-tuples that specify the quasi-independence model $\mathcal{M}_S$. Fix a $j$-coordinate split of $S$ so that $S$ is of the form:
    $\{A_1 +_j B_1, A_2 +_j B_2, \ldots, A_n +_j B_n\}.$
   Define the multisets
    $S^1 = \{\{A_1, A_2, \ldots, A_n\}\}$ and
    $S^2 = \{\{B_1, B_2, \ldots, B_n\}\}.$
    Denote by $S_1$ and $S_2$, the sets formed from the elements of $S^1$ and $S^2$, respectively, without repetition. Then, $S$ is the TFP of $S_1$ and $S_2$ along the $j$th coordinate if and only if:
    \begin{itemize}
        \item for every $A_p$ and $A_r$ in $S^1$ with equal $j$th coordinates, $\#_{S^1}(A_p) = \#_{S^1}(A_r)$, and
        \item for every $B_p$ and $B_r$ in $S^2$ with equal $j$th coordinates, $\#_{S^2}(B_p) = \#_{S^2}(B_r)$.
    \end{itemize}
    $S$ is then a cTFP if and only if this is true for at least one $j\leq k$ and choice of $j$-coordinate split of $S$.
\end{proposition}

First we illustrate this with the example $S$ as in Example \ref{eg1}.
We wish to find coordinate toric fiber factors. Recall that this set is the cTFP along the second coordinate of the following sets of pairs:
\begin{align}
S_1 = \{&(1,1), (1,3), (2,1), (2,2), (3,3)\} \nonumber \\
S_2 = \{&(1,1), (1,3), (2,1), (3,2), (3,3)\} \nonumber
\end{align}
When we form the TFP, we take each $(i,j)$ from $S_1$ and associate it to $(j,k)$ from $S_2$ to give $(i,j,k)$ in $S$. In order to `factorize' $S$ back into sets of pairs, we can attempt to reverse this logic. That is, we split each $(i,j,k)$ along the middle coordinate into $(i,j)$ and $(j,k)$. This results in the following multisets of pairs:
\begin{align}
S^{1} = &\{\{(1, 1), (1, 1), (1, 3), (1, 3), (2, 1), (2, 1), (2, 2), (3, 3), (3, 3)\}\}\nonumber\\
S^{2} = &\{\{(1, 1), (1, 3), (3, 2), (3, 3), (1, 1), (1, 3), (2, 1), (3, 2), (3, 3)\}\}\nonumber
\end{align}
Removing repetitions from these multisets yields models that are the ‘factors' in the toric fiber product to produce $S$. 
The repetitions arise from the fact that $\#_{S^1}(i_1,i_2)$ is equal to the number of pairs in $S^2$ that have $i_2$ as the first coordinate. Hence, in order for the corresponding sets $S_1$ and $S_2$ (without repetitions) to be valid toric fiber factors, we must have $\#_{S^1}(i_1,i_2) = \#_{S^1}(i_1',i_2)$ for each $i_2$ and each pair of $i_1$ and $i_1'$. Similarly, the second bullet point in Proposition \ref{mainlemma} must hold.

Now, consider a different splitting of these triples. For instance, instead of splitting $(i,j,k)$ into $(i,j)$ and $(j,k)$, let us choose to split it into $(i,j)$ and $(i,k)$. Then, we have the following multisets:
\begin{align}
\overline{S^1} = &\{\{(1, 1), (1, 1), (1, 3), (1, 3), (2, 1), (2, 1), (2, 2), (3, 3), (3, 3)\}\}\nonumber\\
\overline{S^{2}} = &\{\{(1, 1), (1, 3), (1, 2), (1, 3), (2, 1), (2, 3), (2, 1), (3, 2), (3, 3)\}\}.\nonumber
\end{align}
Then this splitting does not satisfy the splitting conditions given in Proposition \ref{mainlemma}; for instance, $\#_{\overline S^2}(1,1) = 1 \neq 2 = \#_{\overline S^2}(1,3)$. One can check that $S$ is not the cTFP of $\overline{S_1}$ and $\overline{S_2}$ along the first coordinate.

We are now in a position to provide a proof for the proposition above.

\begin{proof}[Proof of Proposition \ref{mainlemma}]
 Assume that $S$ is a cTFP along the $j$th coordinate. Then, if there exists an $A_p$ and $A_r$ that share the coordinate corresponding to the $j$th coordinate in the set of k-tuples, $\#_{S^1}(A_p)$ corresponds to the number of $B_i$ in the $S^2$ that also have the same value of the original $j$th coordinate. Hence, $\#_{S^1}(A_r) = \#_{S^1}(A_p)$. Similarly, the second bullet point above must hold.
    
    Conversely, let us assume that the conditions stated in the bullet points above hold. Then, the number of $A_p$ in $S^1$ with $j$th coordinate of value $i_j$ is equal to the number of $B_r$ in $S^2$ with $i_j$ as the $j$th coordinate. Hence, each such $A_p$ will concatenate with each such $B_r$ once to form an element of $S$, of the form $A_p +_j B_r$, with $j$th coordinate $i_j$, rendering $S$ as the cTFP of $S^1$ and $S^2$ along the $j$th coordinate. 

    Then, $S$ is a cTFP if and only if this holds for at least one value of $j$, for one splitting.
\end{proof}

In order to motivate the second equivalent condition, we first consider an example of a set of triples which cannot be realized as a cTFP of two multisets of pairs.

\begin{example}
  Consider the following set of triples:
  $S = \{(1,2,1), (1,2,2), (1,1,2), (2,2,2)\}$.
When we split this set of triples along each of the three coordinates, we get, in turn, the following multisets of pairs, where we underline the coordinate that we are splitting along:
\begin{enumerate}
\item $\{\{(\underline{1},2), (\underline1,2), (\underline1,1), (\underline2,2)\}\}$,  $\{\{(\underline1,1), (\underline1,2), (\underline1,2), (\underline2,2)\}\}$
\item $\{\{(1,\underline2), (1,\underline2), (1,\underline1), (2,\underline2)\}\}$,  $\{\{(\underline2,1), (\underline2,2), (\underline1,2), (\underline2,2)\}\}$
\item $\{\{(1,\underline1), (1,\underline2), (1,\underline2), (2,\underline2)\}\}$,  $\{\{(2,\underline1), (2,\underline2), (1,\underline2), (2,\underline2)\}\}$.
\end{enumerate}
None of the three splittings result in valid toric fiber factors, by the conditions given in Propositions \ref{mainlemma}. 

If we consider splitting along the first coordinate, we see that there are $2$ occurrences of $(1,2)$ but only $1$ of $(1,1)$ in the first set of pairs. Similarly, in the second set there is $1$ occurrence of $(1,1)$ but two of $(1,2)$. This from the fact that the triple $(1, 1, 1)$ does not belong to $S$. Indeed, if $(1,1,1)$ were an element of $S$, we would have another copy of $(1,1)$ in the first and second multisets, which would render this a valid cTFP. Similar logic holds for the other two cases.
\end{example}
This example leads to the observation that if $A_1 +_j B_1$ and $A_2 +_j B_2$ occur in $S$, then $A_1 +_j B_2$ and $A_2 +_j B_1$ must as well. Theorem \ref{mainthm} makes this observation precise.

\begin{theorem} \label{mainthm}
    Let $S$ be a set of $k$-tuples. Choose a $j$-coordinate split of $S$ that is associated to some $\mathrm{in(A)}$ and $\mathrm{in(B)}$. Let $S_1 = \{A_1,\dots,A_r\}$ and $S_2 = \{B_1,\dots,B_p\}$ be the sets of tuples resulting from this $j$-coordinate split. Then $S$ is a cTFP according to this $j$-coordinate split if and only if whenever $s_1 = A_1 +^{in(A)}_j B_1$ and $s_2 = A_2 +^{in(A)}_j B_2$ are elements of $S$ with the same $j$th coordinate, the tuples
    $ s_3 = A_1 +^{in(A)}_j B_2$ and $s_4 := A_2 +^{in(A)}_j B_1$ also belong to $S$.

    Then, S is a cTFP if this holds for at least one value of $j \leq k$ and some $j$-coordinate splitting. 
\end{theorem}

\begin{proof} Let $S$ be a set of k-tuples that is associated to a $k$-way quasi-independence model $\mathcal{M}_S$. Then, let $S$ contain the elements $A_1 +_j B_1$ and $A_2 +_j B_2$ such that the two agree in value at the $j$th coordinate. Consider splitting $S$ along the $j$th coordinate into $S_1$ and $S_2$. Then:
    \begin{align}
        &A_1 \text{ and } A_2 \in S_1 \nonumber \\
        &B_1 \text{ and } B_2 \in S_2 \nonumber
    \end{align}
    If $S_1$ and $S_2$ are cTFFs of $S$ that are associated along the $j$th coordinate, then $A_1$ must also be associated to $B_2$ and $A_2$ is also associated to $B_1$. That is, $A_1 +_j B_2$ and $A_2 +_j B_1$ must also be in $S$. 
    
    Conversely, if $S$ containing the elements $A_1 +_j B_1$ and $A_2 +_j B_2$ always implies that $S$ contains $A_1 +_j B_2$ and $A_2 +_j B_1$, we can always split $S$ along the $j$th coordinate into $S_1$ and $S_2$, such that:
    \begin{align}
        &A_1, A_2, A_1 \text{ and } A_2 \in S_1 \nonumber \\
        &B_1, B_2, B_1 \text{ and } B_2 \in S_2 \nonumber
    \end{align}
    That is, $S$ can be split along the $j$th coordinate into valid cTFFs. 
    Then, $S$ is a cTFP if and only if the above holds for some value of $j$ and some splitting.
\end{proof}

We have provided two necessary and sufficient combinatorial conditions for a $k$-way model to be a toric fiber product along one of the coordinates of a given parametrization. As all cTFPs of quasi-independence models can be formed this way, we have provided conditions for a parameterization of a $k$-way quasi-independence model $(k \geq 3)$ to be a cTFP.

\section{2-Way Quasi-Independence Models as Iterated Toric fiber Products}

Here we prove that \emph{every} $2$-way quasi-independence model with rational MLE can be realized as a toric fiber product of linear ideals. We provide a reparameterization of the $2$-way model  and demonstrate that this reparameterization renders the model a TFP. This parameterization is motivated by the known form of the rational MLE characterized by Coons and Sullivant \cite[Theorem~5.4]{coons2021quasi}. The linear factors in the numerator and denominator of this rational function are determined by certain complete bipartite subgraphs of the graph $\mathcal{G}_S$ which we describe below. We then illustrate these concepts in Example \ref{poseteg}.

\begin{definition}
    A set of indices $C = \{i_1, \ldots, i_r\} \times \{j_1, \ldots, j_s\}$ is a \textbf{clique} in $S$ if for all $\alpha \in [r]$ and $\beta \in [s]$, $(i_\alpha , j_\beta ) \in S$. 
A clique in $S$ corresponds to a complete bipartite subgraph of $\mathcal{G}_S$. 
    Let $(i,j)$ be an index in $S$; for brevity, we also represent this pair by $ij$. Then, $\mathbf{Max}(ij)$ is the set of all containment-maximal cliques in $S$ that contain $ij$. The set $\mathbf{Int}(ij)$ is the set of all containment-maximal pairwise intersections of elements of $\mathrm{Max}(ij)$. Similarly, $\mathrm{Max}(S)$ is defined to be the set of all maximal cliques in $S$ and $\mathrm{Int}(S)$ is the set of all maximal intersections of elements of $\mathrm{Max}(S)$. 
\end{definition}

Next we use this clique structure to construct the poset that naturally arises from the rows of the maxmimal cliques, as constructed in \cite[Section~6]{coons2021quasi}. Denote by $\mathrm{rows}(D)$ the set of rows in the star matrix of $S$ that belong to a clique $D$; that is, 
\[\mathrm{rows}(D) = \{i \mid (i,j) \in D \text{ for some } j \in [n] \}.\] 
Similarly, we write
\[
\cols(D) = \{j \mid (i,j) \in D \text{ for some } i \in [m] \}.
\]
We also use this notation to refer to the support of a row or column of the star matrix; that is, if $i \in [m]$,
\[
\cols(i) := \{j \mid (i,j) \in S \}
\]
and we similarly define $\rows(j)$ for $j \in [n]$.
We define the poset $P_S$ whose ground set is $\mathrm{Max}(S)$ and whose elements are ordered by containment of rows. The relation $D < E$ is a cover relation if and only if $D \cap E \in \mathrm{Int}(S)$ \cite[Proposition~6.8]{coons2021quasi}. We note that the proofs and constructions in \cite{coons2021quasi} refer to the subposet $P_S(j)$ consisting of maximal cliques with $j$ in their columns; however, the proofs of all relevant results about $P_S$ are exactly analogous to those in \cite{coons2021quasi}.

Recall from Theorem \ref{janethm}, that a quasi-independence model has ML-degree one if and only if the associated bipartite graph is doubly chordal. In this case, the associated poset $P_S$ is a tree, as shown in \cite[Proposition~6.2]{coons2021quasi}. The poset is not typically graded, but the tree structure allows us to assign to each element of $\mathrm{Max}(S)$ a \emph{level} so that $D < E$ is a cover relation if and only if $\level(D) + 1 = \level(E)$ and so that the minimal level of any maximal clique is 1. We say that a maximal intersection $C \in \mathrm{Int}(S)$ has level $r$ if it is the intersection of two maximal cliques of levels $r$ and $r+1$. 

Define the $0/1$ vectors $a_i$ and $b_j$ to be the the indicator vectors for the $i$th row and $j$th column of the star matrix associated to $S$, respectively. In other words, $a_i$ is $i$th row vector of the first block of $\mathcal{A}_S$ and $b_j$ is the  $j$th row vector of the second block of $\mathcal{A}_S$. We will illustrate these concepts through the following example.  

\begin{example}\label{poseteg}
Let $S=\{11, 12, 13, 21, 22, 24, 25, 31, 32, 44, 45, 55\}$. 
   The  star matrix that represents the corresponding $2$-way quasi-independence model $\mathcal{M}_S$ and its associated bipartite graph $\mathcal{G}_S$ are depicted in Figure \ref{fig:RunningExample}.

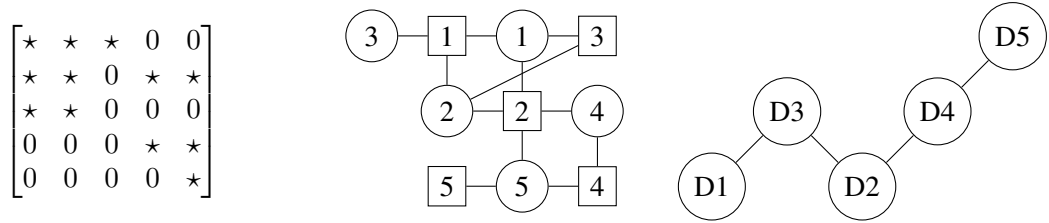
\begin{figure}[!htb]\label{fig:RunningExample}
    \centering
    \begin{minipage}{.3\textwidth}
        \centering
         
       $\begin{bmatrix}
        \star & \star & \star& 0 & 0\\
        \star & \star & 0 & \star &\star\\
        \star & \star & 0 & 0 & 0 \\
        0 & 0 & 0 & \star & \star \\
        0 & 0 & 0 & 0 & \star
        \end{bmatrix}$
    \end{minipage}%
    \begin{minipage}{0.3\textwidth}
    \centering
    \begin{tikzpicture}[every node/.style={minimum size=0.5cm}]
\node [shape=rectangle, draw=black] (A1) at (1,2) {1};
\node [shape=circle, draw=black] (B1) at (2,2) {1};
\node [shape=rectangle, draw=black] (A2) at (2,1) {2};
\node [shape=circle, draw=black] (B2) at (1,1) {2};
\node [shape=circle, draw=black] (B3) at (0,2) {3};
\node [shape=rectangle, draw=black] (A3) at (3,2) {3};
\node [shape=rectangle, draw=black] (A4) at (3,0) {4};
\node [shape=circle, draw=black] (B4) at (3,1) {4};
\node [shape=rectangle, draw=black] (A5) at (1,0) {5};
\node [shape=circle, draw=black] (B5) at (2,0) {5};

\path [-] (A1) edge (B1);
\path [-] (A1) edge (B2);
\path [-] (A1) edge (B3);
\path [-] (A2) edge (B1);
\path [-] (A2) edge (B2);
\path [-] (A2) edge (B4);
\path [-] (A2) edge (B5);
\path [-] (A3) edge (B1);
\path [-] (A3) edge (B2);
\path [-] (A4) edge (B4);
\path [-] (A4) edge (B5);
\path [-] (A5) edge (B5);
            
        \end{tikzpicture}
        \end{minipage}
        \begin{minipage}{0.3\textwidth}
         \begin{tikzpicture} [every node/.style={minimum size=0.4cm}]
        \node [shape=circle, draw=black] (D1) at (0,0) {D1};
        \node [shape=circle, draw=black] (D3) at (1,1) {D3};
        \node [shape=circle, draw=black] (D2) at (2,0) {D2};
        \node [shape=circle, draw=black] (D4) at (3,1) {D4};
        \node [shape=circle, draw=black] (D5) at (4,2) {D5};
        \path [-] (D1) edge (D3);
        \path [-] (D2) edge (D3);
        \path [-] (D2) edge (D4);
        \path [-] (D4) edge (D5);
    \end{tikzpicture}
        \end{minipage}
        \caption{The star matrix, bipartite graph and poset corresponding to $S$ in Example \ref{poseteg}.}
\end{figure}

    We find that there are five maximal cliques: $ D_1 = \{1\} \times \{1,2,3\};
        D_2 = \{2\} \times \{1,2,4,5\}; 
        D_3 = \{1,2,3\} \times \{1,2\}; 
        D_4 = \{2,4\} \times \{4,5\}; 
        D_5 = \{2,4,5\} \times \{5\}.  $
     The maximal intersections of the cliques are: $D_1 \cap D_3 = \{11,12\};
    D_2 \cap D_3= \{21,22\};
    D_2 \cap D_4 = \{24,25\};
    D_4 \cap D_5 = \{25,45\}.$
   The poset $P_S$ on ground set $\{D_1, D_2, D_3, D_4,D_5\}$ ordered by inclusion of rows is also pictured in Figure \ref{fig:RunningExample}.
   
    In this case, the cliques $D_1$ and $D_2$ have level 1, $D_3$ and $D_4$ have level 2 and $D_5$ has level 3. So the level 1 maximal intersections are $D_1 \cap D_3$, $D_2 \cap D_3$ and $D_2 \cap D_4$. The level 2 maximal intersection is $D_4 \cap D_5.$

    The $\Acal$-matrix is as follows:

    { 
     \[\Acal_S = \begin{blockarray}{ccccccccccccc}
     &11 &12 &13 &21 &22 &24 &25 &31 &32 &44 &45 &55\\ 
     \begin{block}{c(cccccccccccc)}
 a_1& 1 & 1 & 1 & 0 & 0 & 0 & 0 & 0 & 0 & 0 & 0 & 0\\
 a_2& 0 & 0 & 0 & 1 & 1 & 1 & 1 & 0 & 0 & 0 & 0 & 0\\
  a_3& 0 & 0 & 0 & 0 & 0 & 0 & 0 & 1 & 1 & 0 & 0 & 0\\
   a_4& 0 & 0 & 0 & 0 & 0 & 0 & 0 & 0 & 0 & 1 & 1 & 0\\
  a_5& 0 & 0 & 0 & 0 & 0 & 0 & 0 & 0 & 0 & 0 & 0 & 1\\
  \cline{2-13}
  b_1 & 1 & 0 & 0 & 1 & 0 & 0 & 0 & 1 & 0 & 0 &0 &0 \\
  b_2 & 0 & 1 & 0 & 0 & 1 & 0 & 0 &0 & 1 & 0 &0 &0 \\
 b_3 & 0 & 0 & 1 & 0 & 0 & 0& 0 & 0 & 0 & 0 & 0 & 0\\
  b_4 & 0 & 0 & 0& 0 & 0 & 1 & 0 & 0 & 0 & 1 & 0 & 0 \\
  b_5 & 0 & 0 & 0 & 0 & 0 & 0 & 1& 0 & 0 & 0 & 1 & 1\\
 \end{block}
\end{blockarray} \ .\]}
\end{example}

The following proposition shows that the partial order on $\max(S)$ induced by reverse inclusion of columns is the same as $P_S$.

\begin{proposition}\label{prop:ReverseInclusionOfColumns}
    Let $D, E \in \max(S)$. Then $\rows(D) \subset \rows(E)$ if and only if $\cols(E) \subset \cols(D)$. Hence $P_S$ is equivalently ordered by reverse inclusion of columns.
\end{proposition}

\begin{proof}
    Let $D$ and $E$ be maximal cliques in $S$ such that $\rows(D) \subset \rows(E)$. 
    Assume for contradiction that $j$ is a column of $E$ that is not a column of $D$.  This implies that there exists some $i$ in the rows of $D$ such that $(i,j) \notin S$. Indeed, if no such $i$ existed, then we could add $j$ to the columns of $D$ to create a larger clique, contradicting maximality of $D$. Since $\rows(D) \subset \rows(E)$ holds, we also have that $i$ is in the rows of $E$. Since $E$ is a clique and $j$ is a column of $E$, we have that $(i,j)\in S$, which is a contradiction. The proof of the other direction is analogous.
\end{proof}

We use these maximal clique, maximal intersection and poset constructions to prove our main result.

\begin{theorem}\label{2waythm} 
    Every $2$-way quasi-independence model has a paramterization that is a coordinate toric fiber product. It then follows that the model is a toric fiber product of linear ideals. 
\end{theorem}

Our goal is to find a new $\mathcal{A}$-matrix, $\Abar_S$, parametrizing $\mathcal{M}_S$ that is a coordinate toric fiber product along all sequential coordinate partitions. The standard $\mathcal{A}$-matrix, $\Acal_S$, of the quasi-independence model will be a submatrix of $\Abar_S$ and the rows that we add will correspond to the indicator vectors of elements if $\int(S)$. Our next step is to show that the indicator vectors of each of these intersections lie within the rowspan of $\mathcal{A}_S$. For any $T \subset S$, let $\Ibb_T \in \{0,1\}^S$ denote the indicator vector for $T$.  We prove this in the following~lemma. 

\begin{lemma}\label{lem1}
Let $S \subset [m] \times [n]$ such that $\Mcal_S$ has ML-degree $1$. Let $C \in \int(S)$. Then $\Ibb_C$ belongs to the rowspan of $\mathcal{A}_S$.
\end{lemma}

Before we prove this lemma, we illustrate the procedure for constructing these indicator vectors of as linear combinations of the rows of $\Acal_S$ in our running example.

\begin{example}\label{poseteg2}
Let S be as in Example \ref{poseteg}.
Consider the maximal intersection $D_1 \cap D_3$. Its indicator vector, $\mathbb{I}_{D_1 \cap D_3}$, has a $1$ in the columns indexed by $11$ and $12$ and $0$ elsewhere. The vector $a_1$ is supported on these two columns but also on $13$. In order to get to the indicator vector, we can subtract $b_3$, which is only supported on $13$, from $a_1$. This results in the desired indicator vector $\mathbb{I}_{D_1 \cap D_3}$. We can use this procedure on the other maximal intersections to write down the indicator vectors:
\begin{align}
    \mathbb{I}_{D_1 \cap D_3} & = a_1 - b_3 \nonumber \\
    \mathbb{I}_{D_2 \cap D_3} & = a_2 - b_4 -b_5 + a_4 + a_5\nonumber \\
    \mathbb{I}_{D_2 \cap D_4} & = b_4 + b_5 -a_4 -a_5 \nonumber \\
    \mathbb{I}_{D_4 \cap D_5} & = b_5 - a_5 \nonumber 
\end{align}

\end{example}

\begin{proof}[Proof of Lemma \ref{lem1}] 
Consider a clique $C = A_1 \times B_0$ such that $C \in \int{S}$.
We can then recursively define for each $k$:
\begin{align*}
    \mathrm{B}_k &= \{j | j \notin \mathrm{B}_{k-1} \text{ and there exists an } i \in \mathrm{A}_{k} \text{ such that } (i,j) \in S\} \\
    \mathrm{A}_{k} &= \{i | i\notin \mathrm{A}_{k-1} \text{ and there exists a } j \in \mathrm{B}_{k-1} \text{ such that } (i,j) \in S\}.
    \end{align*}
    Then we claim that the indicator vector of the clique is equal to
    \begin{equation}\label{indvec}
         \sum_{k=1}^\infty \left[\sum_{i  \in \mathrm{A}_k}a_i - \sum_{j \in \mathrm{B}_k}b_j\right].
    \end{equation}

   In order for this sum to be finite, we need some $\mathrm{A}_k$ or $\mathrm{B}_k$ to be empty. Indeed, if $\mathrm{A}_k$ is empty, then so is $\mathrm{B}_k$ and all $\mathrm{A}_\ell$ and $\mathrm{B}_\ell$ for $\ell > k$. Similarly, if $\mathrm{B}_{k-1}$ is empty, then so are all $\mathrm{A}_\ell$ and $\mathrm{B}_\ell$ for $\ell > k$. 
   We show that all the sets $\mathrm{A}_k$ are disjoint, and hence that one of them is eventually empty as $[m]$ is finite.
   
   First note that for all $k$, $A_k \cap A_{k+1} = B_k \cap B_{k+1} = \emptyset$ by construction. For the sake of contradiction, suppose that $\mathrm{A}_u \cap \mathrm{A}_v$ is non-empty with $v-u \geq 2$. We further assume that $v-u$ is minimal over all  pairs with this property. We will ultimately reach a contradiction by showing that $\mathcal{M}_S$ has ML-degree greater than one by finding an induced cycle or double-square graph in $\mathcal{G}_S$. 
 
    By construction, we have that if $i \in A_k$, then the neighborhood of $a_i$ in $\Gcal_S$ is $B_{k-1} \cup B_k$. Similarly, if $j \in B_k$, then the neighborhood of $b_j$ in $\Gcal_S$ is $A_k \cup A_{k+1}.$ Let $i \in A_u \cup A_v$. Then, there exists a $j \in B_{v-1}$ such that $(i,j) \in S$. Hence, $j$ belongs to $B_{u-1} \cup B_u$, since $i \in A_u$ and $j$ is a neighbor of $i$. If $j \in B_u$, then $B_{v-1} \cap B_u \neq \emptyset$, which contradicts the minimality of $v-u$. So $j \in B_{u-1}$ and $B_{v-1} \cap B_{u-1} \neq \emptyset$. Similarly, since there exists a $j \in B_{u-1} \cap B_{v-1}$, we have that $A_{u-1} \cap B_{u-1} \neq \emptyset.$ Repeatedly applying this argument yields that $B_0 \cap B_{v-u}$ is nonempty.

    Let $j_{v-u} \in B_0 \cap B_{v-u}$. Then by definition, there exist $i_w \in A_w$ and $j_w \in B_w$ for $w$ between $1$ and $v-u$ with $i_wj_w \in S$ for all such $w$ and $i_wj_{w-1} \in S$ for $2 \leq w \leq v-u$. This forms a path of length $2(v-u) -1$ in $\Gcal_S$. But $i_1 \in A_1$ and $j_{v-u} \in B_0$. Since $A_1 \times B_0$ is a clique, we have that $i_1j_{v-u} \in S$ as well. So these vertices form a cycle of length $2(v-u)$. The cycle is chordless by minimality of $v-u$. If $v-u \geq 3$, this contradicts that $\Gcal_S$ is doubly chordal.

    Now suppose $v-u = 2$, so that the subgraph induced by $\{i_1, i_2, j_1, j_2\}$ is a cycle of length $4$. We assumed that $A_1 \times B_0$ is a maximal intersection of maximal cliques. By construction, one of these cliques is $A_1 \times (B_0 \cup B_1)$. The other is of the form $(\tilde{A} \cup A_1) \times B_0$ for some $\tilde{A}$ disjoint from $A_1$. 
    
    By maximality of the intersection $A_1 \times B_0$, we see that the subgraph of $\Gcal_S$ induced by $\tilde{A}$ and $B_1$ is empty.  So in particular, $A_2 \cap \tilde{A} = \emptyset$ and $i_2$ does not belong to $\tilde{A}$. By maximality of the clique $(\tilde{A} \cup A_1) \times B_0$, there must exist a $j' \in B_0$ such that $i_2j' \not\in S$. Moreover, for any $i' \in \tilde{A}$, we have that $i'j_2, i'j' \in S$ by $i'j_1 \not\in S$. So the subgraph of $\Gcal_S$ induced by $\{j', j_1, j_2, i', i_1, i_2\}$ is a double-square. This contradicts that $\Gcal_S$ is doubly chordal bipartite.

    Therefore, $A_k$ and $B_k$ are only nonempty for finitely many $k$.  So the sum in Equation \ref{indvec} is finite and is equal to the indicator vector for $C$ by the construction of $A_k$ and $B_k$. 
\end{proof}

\begin{example}
    To illustrate this proof, consider the maximal intersection $D_2 \cap D_3$ in Example \ref{poseteg}. We have that $D_2 \cap D_3 = \{2\} \times \{1,2\}$, so $A_1 = \{2\}$ and $B_0 = \{1,2\}.$ The columns in row $2$ that are not in $B_0$ are $4$ and $5$, so $B_1 = \{4,5\}$. The rows of these that are not in $A_1$ are $4$ and $5$, so $A_2 = \{4,5\}$. These rows are only supported on columns in $B_1$, so $B_2 = \empty$ and all subsequent $A_i$ and $B_i$ are empty as well. So we recover that
    \[
    \Ibb_{D_2 \cap D_3} = \sum_{i \in A_1} a_i - \sum_{j \in B_1} b_j + \sum_{i \in A_2} a_i = a_2 - b_4 - b_5 + a_4 + a_5.
    \]
\end{example}

Let $h$ be the greatest level of any element of $P_S$. We use Lemma \ref{lem1} to construct the multipartition matrix $\mathcal{\Abar}_S$ with $h+1$ blocks whose rowspan is the same as that of $\mathcal{A}_S$. We construct $\bar{S}$, the state space of the $(h+1)$-way quasi-independence model associated to $\mathcal{\Abar}_S$, so that $\mathcal{A}_{\bar{S}} = \mathcal{\Abar}_S$. The first and last blocks $\mathcal{B}_0$ and $\mathcal{B}_{h}$ of the multipartition matrix are equal to the first and last blocks of $\mathcal{A}_S$; that is, their rows are $\{a_i \mid i \in [m]\}$ and $\{b_j \mid j \in [n]\}$, respectively. The intermediary blocks each correspond to levels of maximal intersections in the poset $P_S$. 

Let $r \in \{0,\dots, h\}$. In order to define the partition matrix $\Bcal_r$, we first define the following three sets. These will index the rows of $\Bcal_r$.
\begin{enumerate}
    \item First, let $X_r$ be the set of all level $r$ maximal intersections; that is,
    $$X_r := \{D \cap E \mid D \cap E \text{ is a maximal intersection in the $r^{th}$ level of } P_S \}.$$
    Since there are no level $0$ or level $h$ maximal intersections, we have that $X_0$ and $X_h$ are empty.
    \item Let $R_0 := [m]$. For each $r \in [h]$, we define
    $$R_r := R_{r-1} \backslash \{i \mid i \text{ is a row in a maximal clique in the } r^{th} \text{ level of } P_S\}.$$
    Note that $R_h$ is the empty set.
    \item Let $C_0 := \emptyset$. For each $r \in [h]$, we recursively define:
    $$C_r := C_{r-1} \cup \{j \mid j \text{ is a column in a level } r \text{ maximal clique but not a level } (r+1) \text{ maximal clique}\}.$$
    Similarly note that $C_h = [n]$.
\end{enumerate}

\begin{definition}\label{matrix_defn}
    We define the partition matrix $\mathcal{B}_r$ whose columns are indexed by elements of $S$ and whose rows as follows. The first rows are the indicator vectors of each maximal intersection in $X_r$. The following rows are each $a_i$ such that $i \in R_r$ and each $b_j$ such that $j \in C_r$.
\end{definition} 

Note that $\Bcal_0$ and $\Bcal_h$ are the first and second blocks of $\Acal_S$ respectively. We define the multipartition matrix $\Abar_S$ to be comprised of partition matrices $\Bcal_0,\dots,\Bcal_h$ in that order. Let
$\bar{S}$ be the set of $(h + 1)$-tuples corresponding to the quasi-independence model specified by $\Abar_S$. We index the entries of $\bar{S}$ as elements of $R_0 \times L_1 \times \ldots \times L_h \times C_h$, where each $L_i, i \in [h]$ is defined as $X_i \cup R_i \cup C_i$. 
We implement the definitions above to form the multipartition matrix in our running example. 

\begin{example}
 We reparameterize the model in Example \ref{poseteg} as the following $4$-way quasi-independence model.
{ \footnotesize
     \[\begin{blockarray}{ccccccccccccc}
     &11 &12 &13 &21 &22 &24 &25 &31 &32 &44 &45 &55\\ 
     \begin{block}{c(cccccccccccc)}
 a_1& 1 & 1 & 1 & 0 & 0 & 0 & 0 & 0 & 0 & 0 & 0 & 0\\
 a_2& 0 & 0 & 0 & 1 & 1 & 1 & 1 & 0 & 0 & 0 & 0 & 0\\
  a_3& 0 & 0 & 0 & 0 & 0 & 0 & 0 & 1 & 1 & 0 & 0 & 0\\
   a_4& 0 & 0 & 0 & 0 & 0 & 0 & 0 & 0 & 0 & 1 & 1 & 0\\
  a_5& 0 & 0 & 0 & 0 & 0 & 0 & 0 & 0 & 0 & 0 & 0 & 1\\
  \cline{2-13}
  D_1\cap D_3& 1& 1& 0& 0& 0& 0& 0 & 0 & 0 & 0 & 0 & 0\\
  D_2\cap D_3& 0& 0& 0& 1& 1& 0& 0 & 0 & 0 & 0 & 0 & 0\\
  D_2\cap D_4& 0 & 0 & 0 & 0 & 0 & 1 & 1 & 0 & 0 & 0 & 0 & 0\\
  b_3 & 0 & 0 & 1 & 0 & 0 & 0& 0 & 0 & 0 & 0 & 0 & 0\\
  a_3& 0 & 0 & 0 & 0 & 0 & 0 & 0 & 1 & 1 & 0 & 0 & 0\\
   a_4& 0 & 0 & 0 & 0 & 0 & 0 & 0 & 0 & 0 & 1 & 1 & 0\\
  a_5& 0 & 0 & 0 & 0 & 0 & 0 & 0 & 0 & 0 & 0 & 0 & 1\\
  \cline{2-13}
  D_4\cap D_5& 0 & 0 & 0 & 0 & 0 & 0 & 1 & 0 & 0 & 0 & 1 & 0\\ 
    b_1 & 1 & 0 & 0 & 1 & 0 & 0 & 0 & 1 & 0 & 0 &0 &0 \\
  b_2 & 0 & 1 & 0 & 0 & 1 & 0 & 0 &0 & 1 & 0 &0 &0 \\
 b_3 & 0 & 0 & 1 & 0 & 0 & 0& 0 & 0 & 0 & 0 & 0 & 0\\
  b_4 & 0 & 0 & 0& 0 & 0 & 1 & 0 & 0 & 0 & 1 & 0 & 0 \\
  a_5& 0 & 0 & 0 & 0 & 0 & 0 & 0 & 0 & 0 & 0 & 0 & 1\\
  \cline{2-13}
  b_1 & 1 & 0 & 0 & 1 & 0 & 0 & 0 & 1 & 0 & 0 &0 &0 \\
  b_2 & 0 & 1 & 0 & 0 & 1 & 0 & 0 &0 & 1 & 0 &0 &0 \\
 b_3 & 0 & 0 & 1 & 0 & 0 & 0& 0 & 0 & 0 & 0 & 0 & 0\\
  b_4 & 0 & 0 & 0& 0 & 0 & 1 & 0 & 0 & 0 & 1 & 0 & 0 \\
  b_5 & 0 & 0 & 0 & 0 & 0 & 0 & 1& 0 & 0 & 0 & 1 & 1\\
 \end{block}
\end{blockarray}\]}
We can use Theorem \ref{mainthm} to show that this is a cTFP along its second and third coordinates.
\end{example}

We now define important cliques associated to each row and column.
For each $i \in [m]$, define
       \[
       E_i := \{j \in [m] \mid \cols(i) \subset \cols(j) \} \times \cols(i).
       \]
Similarly, for each $j \in [n]$, define
       \[
       E^j := \rows(j) \times \{i \in [n] \mid \rows(j) \subset \rows(i) \}.
       \]
We illustrate this in the case of the running example, Example \ref{poseteg}. We have $E_i = D_i$ for each $i = 1,\dots,5$. We have $E^1 = E^2 = D_3$, $E^3 = D_1$, $E^4 = D_4$ and $E^5 = D_5.$

\begin{proposition}\label{rmk:meetandjoin}
Let $i \in [m]$ and $j \in [n]$. The cliques $E_i$ and $E^j$ belong to $\max(S)$. Moreover,
\begin{enumerate}
    \item for each $D \in \max(S)$ with $i \in \rows(D)$, we have $E_i \leq D$ in $P_S$ and
    \item for each $D \in \max(S)$ with $j \in \cols(D)$, we have $E^j \geq D$ in $P_S$.
\end{enumerate}
\end{proposition}

\begin{proof}
We prove these results in the case of $E_i$.
       First, note that all rows that are supported on $\cols(i)$ are included in $E_i$. We cannot add another column because we have included all columns that are supported on row $i$. So $E_i \in \max(S)$.
       
       Let $D \in \max(S)$ with $i \in \rows(D)$. Since $i \in \rows(D)$, we have $\cols(D) \subset \cols(i)$. For all $j \in \rows(E_i)$, we have $\cols(D) \subset \cols(i) \subset \cols(j)$ by definition. So the maximality of clique $D$, $j \in \rows(D)$. So $E_i \leq D$ in $P_S$.

       The analogous results hold for $E^j$ by considering $P_S$ as being ordered by reverse inclusion of columns as described in Proposition \ref{prop:ReverseInclusionOfColumns}.
\end{proof}

\begin{lemma}\label{preplemma}
   Let $x \in [m]$. Let $D_1, D_2 \in \max(S)$ such that $x \in \rows(D_1)$ and $x \in \rows(D_2)$. Let $\level(D_1) = r < s = \level(D_2)$. Then for each $i$ such that $r \leq i \leq s$, there exists a $D \in \max(S)$ such that $x \in \rows(D)$  and $\level(D) = i$.
\end{lemma}

\begin{proof}
    By Proposition \ref{rmk:meetandjoin}, we have $E_i \leq D_1, D_2$ in $P_S$. Let $\level(E_i) = t$ so that $t \leq r$. By definition of the level of a clique, we have that the chain from $E_i$ to $D_2$ contains cliques of all levels between $t$ and $s$, so in particular, it contains a clique of level $r$. 
\end{proof}

We can now show that $\mathcal{\Abar}_S$ is a multipartition matrix.

\begin{lemma}\label{lem2}
    Let $\Mcal_S$ be a $2$-way quasi-independence model with rational MLE. The matrix $\mathcal{B}_r$, as defined in Definition \ref{matrix_defn} is a partition matrix, and hence $\mathcal{\Abar}_S$ is a multipartition matrix. 
\end{lemma}
\begin{proof}
    First we show that there is at most one $1$ in each column in each block of $\mathcal{\Abar}_S$ which implies that rows of $\mathcal{B}_r$ have disjoint support. Indeed, none of the rows associated to elements of $X_r$ overlap in their support since cliques in the same level of $P_S$ are disjoint. If $i \neq j \in [m]$, then $a_i$ and $a_j$ have disjoint support. Similarly if $i \neq j \in [n]$, then $b_i$ and $b_j$ have disjoint support. 
    
    If $i \in R_r$, then $i$ is not a row of a level $r+1$ maximal clique, and hence cannot be included in any level $r$ maximal intersection. Therefore the support of the $a_i$s is disjoint from that of $\Ibb_D$ for each $D \in X_r$. Similarly if $j \in C_r$, the $j$ is not a column of a level $r+1$ maximal clique. So $j$ is not a column of any level $r$ intersection. This implies that the support of the $b_j$s is disjoint from that of the maximal intersection vectors. 

    It is now left to show that the rows $a_i$s and $b_j$s have disjoint support when $i \in R_r$ and $j \in C_r$. Assume, for contradiction that there is an $i \in R_r$ and $j \in C_r$ such that the support of these vectors intersect in a non-empty set. Then $ij \in S$ and $ij$ is not in a level $r$ maximal intersection. Furthermore, since $i \in R_r$, the row $i$ is not in any level $r$ maximal clique. So $\level(E_i) > r$. Since $j \in C_r$, the column $j$ is not in any level $r+1$ maximal clique. So $\level(E^j) \leq r$. But by Proposition \ref{rmk:meetandjoin}, we must have $E_i < E^j$ in $P_S$, which is a contradiction.


    So the rows of $\mathcal{B}_r$ have disjoint support. 
    To complete the proof of this lemma, we must show that each column in $\mathcal{B}_r$ has a $1$ in some row. If the column indexed by $ij$ in $\Bcal_r$ has a $1$ in one of its rows, we say that $ij$ is \emph{covered} in $\Bcal_r$. 
    
    Let $xz \in S$. If $x \in R_r$ or $z \in C_r$, then $xz$ is covered in $\Bcal_r$. Suppose otherwise. Then since $x \not\in R_r$, we have that $x \in \rows(E)$ for some $E \in \max(S)$ with $\level(E) \leq r$. In particular, $\level(E_x) \leq r$. Since $z \not\in C_r$, there exists a $D \in \max(S)$ such that $z \in \cols(D)$ and $\level(D) = r+1$. Hence, $\level(E^z) > r.$ 

    Since $x \in \rows(E_x)$ and $x \in \rows(E^z)$, we have that $E_x < E^z$ in $P_S$. So there exist $D_r, D_{r+1} \in \max(S)$ of levels $r$ and $r+1$ respectively such that $E_x \leq D_r < D_{r+1} \leq E^z$. So $D_r \cap D_{r+1}$ is a level $r$ intersection. Moreover, since $z \in \cols(E_x), \cols(E^z)$, we have that $z$ belongs to $\cols(D_r)$ and $\cols(D_{r+1})$ as well. So $xz \in D_r, D_{r+1}$. Hence, $xz \in X_r$. So $xz$ is covered in $\Bcal_r$, as needed.
    \end{proof}

We use these lemmas to prove the theorem stated at the beginning of this section; Theorem \ref{2waythm}. 

\begin{proof} [Proof of \ref{2waythm}]
    Suppose that $P_S$ has $h$ levels so that $\mathcal{\Abar}_S$ is a multipartition matrix with blocks $\Bcal_0,\dots,\Bcal_h$. 
    Here, we will prove that the set of coordinates associated to $\mathcal{\Abar}_S$ is a cTFP along the $r$th coordinate for each $r = 1, \ldots, h-1$. We will do this by using the conditions given in Theorem \ref{mainthm}. Let $ij, k\ell \in S$ and let $(i, c_1,\ldots, c_{h-1}, j)$ and $(k, d_1, \ldots, d_{h-1}, \ell)$ be their corresponding elements of $\bar{S}$. We must show that if $c_r=d_r=x$, then 
    \begin{equation}\label{eq:neededtuples}
    (i, c_1, \ldots, c_{i-1}, x, d_{i+1}, \ldots, d_{h-1}, \ell) \quad \text{and} \quad (k, d_1, \ldots, d_{i-1}, x, c_{i+1}, \ldots c_{h-1}, j)\end{equation} are also both in $\bar{S}$. 
    There are three cases depending on whether $x \in R_r$, $x \in C_r$ or $x \in X_r$.

    \noindent \textbf{Case I:} Let $x \in R_r.$ Then we have $i = k = x$. Then by definition, $x \in R_s$ for all $s \leq r$. So for all $s \in \{0,\dots,r\}$ we have $c_s = d_s = x$.  So the tuples in Equation \ref{eq:neededtuples} lie in $\bar{S}$ trivially.

    \noindent \textbf{Case II:} Similarly, let $x \in C_r$. Then $j = \ell = x$ and by definition, $x \in C_s$ for all $s \geq r$. So for all $s \in \{r, \dots, h\}$ we have $c_s = d_s = x$. So the tuples in Equation \ref{eq:neededtuples} lie in $\bar{S}$ trivially.

   \noindent \textbf{Case III:} Let $x \in X_r$ so that $x = D_p \cap D_q := M_{pq}$ where $\level(D_p) = r$, $\level(D_q) = r+1$ and $D_p \leq D_q$ is a cover relation in $P_S$. There are three subcases.
   
   \noindent\textit{Subcase (a):} Suppose that $D_p$ consists of a single row. Then all elements of $M_{pq}$ have the same first coordinate and in particular, $i = k$. Moreover, $D_p$ is the minimal element of $P_S$ containing row $i$ and by definition of $R_s$, we must have $i \in R_s$ for each $s < r$. So $c_s = d_s = i$ for each $s < r$. 
   Hence, the parts of the tuples preceding $M_{pq}$ are equal, so the tuples in Equation \ref{eq:neededtuples} belong to $\bar S$.

   \smallskip
   
   \noindent\textit{Subcase (b):} Similarly, we consider the case where $D_q$ consist of a single column so that all elements of $M_{pq}$ have the same second coordinate and $j = \ell.$  Since $P_S$ is also ordered by reverse inclusion of columns, we must have that $D_q$ is the maximal element of $P_S$ containing column $j$. So we have $j \in C_s$ and $c_s = d_s = j$ for all $ s > r$.  Hence, the parts of the tuples following $M_{pq}$ are equal, so the tuples in Equation \ref{eq:neededtuples} belong to $\bar S$ as needed.

  \smallskip
   
    \noindent\textit{Subcase (c):} Finally we consider the case where $D_p \cap D_q$ are not contained in a single row or column. Since $ij, k\ell \in M_{pq}$ and $M_{pq}$ is a clique, we also have that $i\ell, kj \in M_{pq}$.    Let the tuple in $\bar S$ corresponding to  $i\ell$ be of the form: $(i, e_1, \ldots, e_{r-1}, M_{pq}, e_{r+1}, \ldots, e_n, \ell)$. Then, we need to show that:
   \begin{enumerate}
       \item $c_s = e_s, \text{ for } s = \{1, \ldots, r-1\} $
       \item $d_s = e_s, \text{ for } s = \{r+1, \ldots, h-1\}$
   \end{enumerate}



   To prove the first point, we will show that for any $E \in \max(S)$ such that $\level(E) \leq \level(D_p),$ if $(i,j)\in E$ then $(i,\ell) \in E$ as well. First note that since $ij \in E$ and $ij \in D_p$, we must have that $E \leq D_p$ in $P_S$. 
   Since $E \leq D_p, D_q$ in $P_S$, we have $\rows(E) \subseteq \rows(D_p\cap D_q)$. By Proposition \ref{prop:ReverseInclusionOfColumns}, $\cols(D_p\cap D_q) \subseteq \cols(E)$. Now, $j, \ell \in \cols(D_p\cap D_q)$, which implies that $j,\ell \in \cols(E)$. Also, $i \in \rows(E)$. Hence, $i\ell \in E$. 

   Similarly, to prove the second point, we show that for $E \in \max(S)$ with $\level(E) \geq \level(D_q)$ in $P_S$, if $k\ell \in E$, then $i\ell \in E$. Since $k\ell \in E$ and $k \ell \in D_q$, we must have that $D_q \leq E$ in $P_S$. Since $E \geq D_q$, we have $\rows(D_p\cap D_q) \subseteq \rows(C)$. Now, $i, k \in \text{rows}(D_p\cap D_q)$, and hence $i,k \in \text{rows}(E)$. Also, $\ell$ is a column of $E$. Hence, $i\ell \in E$. 

   Therefore, the tuple in $\bar S$ corresponding to $i\ell$ is $(i, c_1, \ldots, c_{r-1}, M_{pq}, d_{r+1}, \ldots, d_h, \ell)$. We can similarly prove that tuple corresponding to $kj$ is $(k, d_1, \ldots, d_{r-1}, M_{pq}, c_{r+1}, c_h, j)$. Therefore, the conditions stated in Theorem \ref{mainthm} are fulfilled in this case.



    Hence, in each case, the conditions of Theorem \ref{mainthm} are satisfied. So, the matrix $\mathcal{\Abar}_S$ is a cTFP along the $r$th coordinate for each $r = 1, \ldots, h-1$. 
\end{proof}

A consequence of this theorem is that all 2-way quasi-independence models with rational MLE are ``built" out of very simple building blocks via the toric fiber product. In fact, it follows that their vanishing ideals can be obtained by repeatedly taking toric fiber products of linear ideals.

\begin{definition}
For each $r \in \{0,\dots,h+1\}$, denote by $\Abar_S^{r}$ the multipartition matrix with blocks $\Bcal_0,\dots,\Bcal_{r}$. We say that $\Abar_S$ is an \textbf{iterated toric fiber product of linear ideals} if for each $r$, there exist matrices $X$ and $Y$ such that the columns of $X$ are a sub-multiset of those of $\Abar^{r}$, the columns of $Y$ are a sub-multiset of those of $\Bcal_{r+1}$, and $\Abar_S^{r+1}$ is a toric fiber product of $X$ and $Y$. 
\end{definition}

This a rephrasing of the definition introduced by Coons, Langer, and Ruddy \cite[Definition~3.19]{coons2022rational}.

\begin{theorem}\label{thm:LinearIdeals}
    Let $\Mcal_S$ be a $2$-way quasi-independence model with rational maximum likelihood estimator. Let $\Abar_S$ be as in Lemma \ref{lem2}. Then $\Abar_S$ is an iterated toric fiber product of linear ideals.
\end{theorem}

\begin{proof} 
    Since $\Abar_S$ is a cTFP along its $r$th coordinate, so is $\Abar_S^{r+1}$. In fact, we will show that we can obtain $\Abar_S^{r+1}$ as a toric fiber product of two matrices whose columns are a sub-multiset of those of $\Abar_S^r$ and $\Bcal_{r+1}$ respectively. To do this we will construct a set of $(r+1)$-tuples,
    \[
    T \subset L_0 \times L_1 \times \dots \times L_r
    \]
    and a multiset of singletons $T' \subset L_{r+1}$. We then give a procedure for appending the elements of $T'$ to those of $T$ and show that this is a toric fiber product.

    Let $\bar{S}^r$ be the multiset obtained by truncating each element of $\bar S$ after block $\Bcal_r$. This is the multiset of tuples that yields the multipartition matrix $\Abar_S^{r+1}$. Let $T$ be the set underlying $\bar{S}^r$; that is, we remove all repetitions from $\bar{S}^r$ to create $T$. 
    To construct $T'$, we must specify the multiplicity of each element of $L_{r+1}$ in $T'$. If $i \in R_{r+1}$, its multiplicity in $T'$ is $\#\cols(i)$.  If $E \in X_{r+1}$ is a maximal intersection, its multiplicity in $T'$ is $\#\cols(E).$ If $j \in C_{r+1}$, its multiplicity in $T'$ is $1$.

    Let $Y_r$ denote the set of all level $r$ elements of $\max(S)$. We partition the elements of $T$ into parts $G_x$ indexed by $x \in  R_{r+1} \cup C_r \cup Y_{r+1}$ based on their last coordinate as follows. Let $A \in T$. If $A_r = i \in R_{r+1}$, then $A$ belongs to $G_i$. If $A_r = i \in R_r \setminus R_{r+1}$, then the minimal clique in $P_S$ with $i$ in its rows, $E_i$, is a level $r+1$ clique. So we place $A$ in $G_{E_i}$. If $A_r = j \in C_r$, then $A$ belongs to $G_j$. If $A_r = D \cap E \in X_r$ with $\level(D) = r$ and $\level(E) = r+1$, then $A$ belongs to $G_E$.

    We similarly partition the multiset $T'$ into parts $H_x$ indexed by $x \in R_{r+1} \cup C_r \cup Y_{r+1}.$ Let $A \in T'$. If $A = i \in R_{r+1}$, then $A$ belongs to $H_i$. 
    If $A = D \cap E$ with $\level(D) = Y_{r+1}$ and $\level(E) = Y_{r+2}$, then $A$ belongs to $G_D$. 
    
    Finally we must partition the elements $A \in T'$ such that $A \in C_{r+1}$. If $A = j \in C_r$, then $A$ belongs to $H_j$. Otherwise, by definition of $C_{r+1}$, there exists a clique $E \in \max(S)$ with $\level(E) = r+1$ such that $j \in \cols(E)$ but $j \not\in \cols(D)$ for any clique with $\level(D) = r+2$. In other words, $E = E^j$. If $E^j \in Y_{r+1}$, then $A$ belongs to $H_{E^j}$. Otherwise, $E^j$ is minimal in $P_S$, so $E^j = E_i$ for some row $i$ with $i \in R_r$. In this case, $A$ belongs to $H_i.$

    By construction, we have that $\bar{S}^{r+1}$ is obtained from $T$ and $T'$ by appending all elements of $H_x$ to all elements of $G_x$ for each $x \in R_r \cup C_r \cup Y_{r+1}.$ Let $\tilde{\Acal}^r$ and $\tilde{\Bcal_{r+1}}$ denote the $A$-matrices for the multipartition models specified by $T$ and $T'$ respectively. Let $\tilde{\Bcal}_r$ denote the last block of $\tilde{\Acal}^r$, which corresponds to the $r$th coordinates of elements of $T$. In order to show that $\bar{S}^{r+1}$ is the toric fiber product of $T$ and $T'$, we must show 
    that the toric ideals $I_T$ and $I_{T'}$ are homogeneous with respect to the multigradings induced by the partitions $G_x$ and $H_x$ respectively. 
    
    Let $I_T \subset \mathbb{C}[y_A \mid A \in T]$. Since $I_T$ is toric, it suffices to check that each of its binomials are homogeneous with respect to this multigrading. Let $\bfy^\alpha - \bfy^\beta \in I_T$. 
    By \cite[Corollary~4.3]{sturmfels1996grobner}, we have $\alpha - \beta \in \ker \tilde{\Acal}^r$. This binomial is homogeneous with respect to the multigrading if and only if
    \[
    \sum_{A \in G_x} \alpha_A = \sum_{A \in G_x} \beta_A
    \]
    for each $x \in R_{r+1} \cup C_r \cup Y_{r+1}$. In other words, we must check that $(\alpha - \beta) \cdot \Ibb_{G_x} = 0$ for all such $x$. 
    Since $\alpha - \beta \in \ker \tilde{\Acal}^r$, we may prove this by showing that the indicator vector for $G_x$ lies in the rowspan of  $\tilde{\Acal}^r$  for each $x \in R_{r+1} \cup C_r \cup Y_{r+1}$. Similarly, we must show that the indicator vector for $H_x$ lies in the rowspan of $\tilde{\Bcal}_{r+1}$ for each such $x$.

    This is true by our constructions of $G_x$ and $H_x$. First, consider the set $T$. Our partition of elements into the parts $G_x$ depends only on the last coordinate. If $A, B \in T$ have $A_r = B_r$ and $A \in G_x$, then $B \in G_x$ as well. Denote the rows of  $\tilde{\Bcal}_{r}$ by $\alpha_y$ such that $y \in L_r$. Then the indicator vector of $G_x$ is the sum of those $\alpha_y$ such that $y \in G_x$. Similarly, if $A,B \in T'$ are equal, then $A \in H_x$ if and only if $B \in H_x$. Let $\beta_y$ for $y \in L_{r+1}$ denote the rows of $\tilde{\Bcal_{r}}$. Then the indicator vector for $H_x$ is the sum of those $\beta_y$ such that $y \in H_x$. Hence the corresponding toric ideals are homogeneous with respect to the multigradings induced by these partitions.

    So we conclude that $\bar{S}^{r+1}$ is the toric fiber product of $T$ and $T'$. The toric ideal corresponding to $T'$ is a linear ideal. Since this holds for all $r \in \{0,\dots,h\}$, we see that $\bar{S}$ is an iterated toric fiber product of linear ideals. 
\end{proof}

\begin{example}
    Recall the quasi-independence model from Example \ref{poseteg2}. Then, $\bar{S}$ is a set of $4$-tuples. In this case, we will illustrate the proof of Theorem \ref{thm:LinearIdeals} by showing that $\Acal^2_S$ is a TFP of two matrices whose columns are sub-multisets of the columns of $\Acal^1_S$ and $\Bcal_2$, respectively. In order to differentiate between rows and columns of the star matrix, we denote elements of $R_r$ by $a_i $ such that $ i \in [m]$ and elements of $C_r$ by $b_j $ such that $j \in [n]$. 

    Here, $T \subset L_0 \times L_1$, is the set $\bar{S}^1$ without repetitions:
    $$T = \{(a_1,D_1\cap D_3), (a_1,b_3), (a_2,D_2\cap D_3), (a_2,D_2 \cap D_4), (a_3,a_3), (a_4,a_4), (a_5,a_5)\}.$$
    Furthermore, $L_2 = \{\{b_1, b_2, b_3, b_1, b_2, b_4, D_4\cap D_5, b_1, b_2, b_4, D_4\cap D_5, a_5\}\}$. We note that $\#\text{cols}(5)$ is $1$ and hence hence the multiplicity of $a_5$ in $T'$ is one. The maximal intersection has one column, and so $D_4\cap D_5$ has a multiplicity of one in $T'$. Finally, the elements of $C_2$ each appear once in $T'$. Hence $T' = \{D_4 \cap D_5, b_1, b_2, b_3, b_4, a_5\}$. 

    The index set of the partitions $G_x$ and $H_x$ is 
    \[
    R_2 \cup C_1 \cup Y_2 = \{a_5, b_3, D_3, D_4\}.
    \]
    
    For example, consider $(a_3,a_3)$. We have that $3 \in R_2\setminus R_1$. Hence, $(a_3,a_3)\in G_{E_3}$. Note that $E_3 = D_3$. Similarly, we find that $(a_4,a_4) \in G_{E_4} = G_{D_4}$. Also, $3 \in C_{2} \subseteq C_1$. So, $a_3 \in H_3$. Then, noting that $E^1=E^2=D_3$ and that $E^4 = D_4$, we see the following:

    \begin{align*}
         &G_3 = (a_1,b_3), && H_3  = b_3; \\
         &G_5 = (a_5,b_5), && H_5 = a_5; \\
         &G_{D_3} = \{(a_1,D_1\cap D_3), (a_2,D_2\cap D_3), (a_3,a_3)\}, && H_{D_3}  = \{b_1, b_2\};\\
          & G_{D_4} = \{(a_2,D_2\cap D_4), (a_4, a_4)\}, && H_{D_4} = \{D_4 \cap D_5, b_4\}.
    \end{align*}

    Appending each $k \in H_x$ to each $ij \in G_x$, we obtain the set of triples 
    \begin{align*}
        \{&(a_1,D_1\cap D_3,b_1), (a_1,D_1\cap D_3,b_2), (a_1,b_3,b_3), (a_2,D_2\cap D_3,b_1), (a_2,D_2\cap D_3,b_2), (a_2,D_2\cap D_4,b_4),\\
        &(a_2,D_2\cap D_4,D_4\cap D_5), (a_3,a_3,b_1), (a_3,a_3,b_2), (a_4,a_4,b_4), (a_4,a_4,D_4 \cap D_5), (a_5,a_5,a_5)\}.
    \end{align*} 

    This set of triples is indeed $\bar{S}^2$.
\end{example}

This result relates to the  the classical iterative proportional scaling (IPS) algorithm which can be applied to a multipartition model to approximate its maximum likelihood estimate. As discussed in \cite{coons2022rational}, the performance of IPS depends on the $A$-matrix of the model and can differ with different parametrizations. In special cases, IPS may produce the \emph{exact} MLE in only one cycle; such a parametrization is said to exhibit \emph{one-cycle convergence}. The following is a direct consequence of our Theorem \ref{thm:LinearIdeals} and \cite[Theorem~3.22]{coons2022rational}.

\begin{corollary}
    Every $2$-way quasi-independence model has a parametrization for which the iterative proportional scaling algorithm exhibits one-cycle convergence.
\end{corollary}

\section{Lawrence Lifts of $2$-Way Quasi-Independence Models}

In this section, we review Lawrence lifts and consider their application to quasi-independence models. Lawrence lifts were first introduced in \cite{orientedmatroids} in the context of oriented matroids. In our context, they are an operation one can perform on a multipartition matrix to obtain another multipartition matrix with one more block. Several papers in the algebraic statistics literature make use of the Lawrence lift operation. For instance, they are used in Bernstein and Sullivant's classification of unimodular binary hierarchical models in \cite{bernstein2017unimodular}. In \cite{brysiewicz2023lawrence}, the authors consider the quasi-independence models specified by Lawrence lifts of certain unimodular matrices and compute their maximum likelihood degrees. 
In this section, we provide a necessary and sufficient condition for the model generated from the Lawrence lift of a $2$-way quasi-independence model to be a cTFP. 


Below, we introduce the concept of a Lawrence lift and that of a modified Lawrence lift. The latter will allow us to translate the lift into the context of quasi-independence models. 
\begin{definition}
    The \textbf{Lawrence lift} of a matrix $\mathcal{T}$ with $m$ rows and $n$ columns is the matrix $\Lambda(\mathcal{T})$, written as :
    $$\begin{pNiceArray}{c:c} 
\mathcal{T} & \bf{0}_{m \times n}\\
\hline
\bf{0}_{m \times n} & \mathcal{T}\\
\hline
\bf{I}_n & \bf{I}_n\\
\end{pNiceArray}$$
\end{definition}

We specifically consider the case where $\mathcal{T}$ is a  multipartition matrix representing a $2$-way quasi-independence model. 
Consider the following multipartition matrix $\mathcal{T} = (\mathcal{A}|\mathcal{B})^T.$
    Let $\mathcal{A}$ be of dimension $a \times n$ and $\mathcal{B}$ be of $b \times n$. Then, we can write $\Lambda(\mathcal{T})$ as:
    $$\begin{pNiceArray}{c:c} 
\mathcal{A} & \bf{0}_{a \times n}\\
\mathcal{B} & \bf{0}_{b \times n}\\
\hline
\bf{0}_{a \times n} & \mathcal{A}\\
\bf{0}_{b \times n} & \mathcal{B}\\
\hline
\bf{I}_n & \bf{I}_n\\
\end{pNiceArray}$$

The above is not a multipartition matrix, as the sums of the columns in each of the partitions does not sum to $1$. However, we are able to rearrange rows without affecting the resultant quasi-independence model. The following rearrangement of the above results in a multipartition matrix that represents a $3$-way quasi-independence model:
\begin{equation}\label{law_lift}
   \Lambda'(\mathcal{T})= \begin{pNiceArray}{c:c}
\mathcal{A} & \bf{0}_{a \times n}\\
\bf{0}_{a \times n} & \mathcal{A}\\
\hline
\mathcal{B} & \bf{0}_{b \times n}\\
\bf{0}_{b \times n} & \mathcal{B}\\
\hline
\bf{I}_n & \bf{I}_n\\
\end{pNiceArray}
\end{equation}

In the case of a multipartition matrix, we denote this as a \textbf{modified Lawrence lift}. Note that the rowspans of $(\Lambda(\mathcal{T}))$ and $(\Lambda'(\mathcal{T}))$ are the same, which implies that the models associated to them are the same. To illustrate this construction, consider the following example. 

\begin{example} \label{lleg}
    Consider the following bipartite graph $\mathcal{G}_S$, which we note is a tree, and its associated $\mathcal{A}$-matrix $\mathcal{A}_S$:

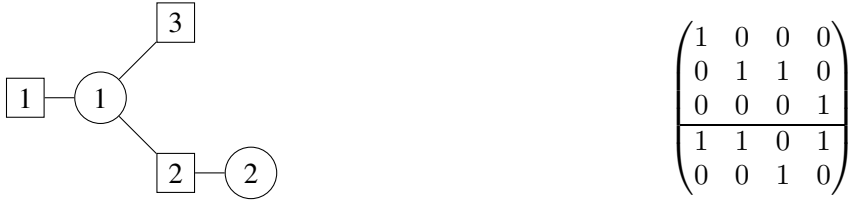
\begin{figure}[!htb]
    \centering
    \begin{minipage}{.5\textwidth}
        \centering
      \begin{tikzpicture}[every node/.style={minimum size=0.5cm}]
\node [shape=rectangle, draw=black] (A1) at (0,1) {1};
\node [shape=circle, draw=black] (B1) at (1,1) {1};
\node [shape=rectangle, draw=black] (A2) at (2,0) {2};
\node [shape=rectangle, draw=black] (A3) at (2,2) {3};
\node [shape=circle, draw=black] (B3) at (3,0) {2};
\path [-] (A1) edge (B1);
\path [-] (B1) edge (A3);
\path [-] (B1) edge (A2);
\path [-] (A2) edge (B3);
\end{tikzpicture}
    \end{minipage}%
    \begin{minipage}{0.5\textwidth}
    \centering
    $$\begin{pmatrix}
1&0&0&0\\
0&1&1&0\\
0&0&0&1\\
\hline
1&1&0&1\\
0&0&1&0\\
\end{pmatrix}$$
        \end{minipage}
\end{figure}

We can then take the modified Lawrence lift of this matrix:
{\footnotesize
$$\begin{pNiceArray}{cccc:cccc}
1&0&0&0&0&0&0&0\\
0&1&1&0&0&0&0&0\\
0&0&0&1&0&0&0&0\\  
\hdottedline
0&0&0&0&1&0&0&0\\
0&0&0&0&0&1&1&0\\
0&0&0&0&0&0&0&1\\
\hline
1&1&0&1&0&0&0&0\\
0&0&1&0&0&0&0&0\\
\hdottedline
0&0&0&0&1&1&0&1\\
0&0&0&0&0&0&1&0\\
\hline
1&0&0&0&1&0&0&0\\
0&1&0&0&0&1&0&0\\
0&0&1&0&0&0&1&0\\
0&0&0&1&0&0&0&1\\
\end{pNiceArray}$$}
The set of triples that represents the $3$-way quasi-independence model is:
$$S = \{(1,1,1), (2,1,2), (2,2,3), (3,1,4), (4,3,1), (5,3,2), (5,4,3), (6,3,4)\}.$$
This set is not a cTFP, as it does not satisfy the conditions given in Theorem~\ref{mainthm}. Specifically, we cannot split along the first coordinate as $S$ contains $(2,1,2)$ and $(2,2,3)$ but neither $(2,2,2)$ nor $(2,1,3)$. Similarly, we cannot split along the second as $S$ contains $(1,1,1)$ and $(2,1,2)$ but does not contain $(1,1,2)$ or $(2,1,1)$. Finally, it is impossible to split along the third coordinate as there is $(2,1,2)$ and $(5,3,2)$ but neither $(5,1,2)$ nor $(2,3,2)$. 
\end{example}

Let us invoke the definition of the following type of graph. 

\begin{definition}
    A \textbf{star tree} with $k+1$ vertices is a tree with one central node and $k$ leaves. 
\end{definition}

We consider examples such as the one above to conclude the following. 

\begin{theorem} \label{llthm}
    Let $\mathcal{M}_S$ be a $2$-way quasi-independence model with associated $\mathcal{A}$-matrix $\mathcal{A}_S$ and bipartite graph $\mathcal{G}_S$. Let $\mathcal{M}_{S'}$ be the three-way quasi-independence model with $\Acal$-matrix $(\Lambda'(\mathcal{A}_S))$. The modified Lawrence lift $(\Lambda'(\mathcal{A}_S))$ is a cTFP if and only if $\mathcal{G}_S$ is a collection of stars, such that the central vertex of each star corresponds to the same variable. 
\end{theorem}

\begin{proof}
Consider a multipartition matrix $(\mathcal{A|B})^T$ that is associated to a set of pairs $S$ and a quasi-independence model $\mathcal{M}_S$, with associated graph $\mathcal{G}_S$. Let $S$ be formed of the pairs $(a_1,b_1), (a_2,b_2), \ldots, (a_n,b_n)$.

Recall the definition of $\Lambda'((\mathcal{A|B})^T)$ from Equation \ref{law_lift}. Let this matrix be associated to the quasi-independence model $\mathcal{M}_{S'}$. In other words, $\Lambda'((\mathcal{A|B})^T)$ is $\mathcal{A}_{S'}$ Then, $S'$ is:
$$\{(a_1, b_1, 1), (a_2, b_2, 2), \ldots, (a_n, b_n, n), (a_1 + a, b_1+b, 1), (a_2+a, b_2+b, 2), \ldots, (a_n+a, b_n+b, n)\}. $$

This set can never be a TFP along the third coordinate, as we have $(a_i, b_i, i)$ and $(a_i+i, b_i+i, i)$ in $S'$ but neither $(a_i+i, b_i, i)$ nor $(a_i, b_i+i, i)$ in $S'$, for each $i \in \{1, \ldots, n\}$, by construction. Then, by the conditions provided in Theorem \ref{mainthm}, $S'$  is not a cTFP along the third coordinate. 

Now, assume that $\mathcal{A}_{S'}$ is a cTFP. Then, it can only be a cTFP along its first, or its second coordinate.

First assume that $\mathcal{A}_{S'}$ is a cTFP along the first coordinate. Assume for contradiction that $a_i = a_j$, for some $i \neq j$. Let $a_i = a_j = \alpha$. Then, the triples $(\alpha, b_i, i), (\alpha, b_j, j)$ are in $S'$.
Now, in this Lawrence lift, the only triples that have $i$ as the third coordinate are $(\alpha, b_i,i)$ and $(\alpha+i, b_i+i, i)$. So, $(\alpha, b_j, i)$ cannot exist in $S'$. Similarly, $(\alpha, b_i, j)$ cannot exist within $S'$ either. This contradicts the fact that $\mathcal{A}_{S'}$ is a cTFP along the first coordinate. 

If $\mathcal{A}_{S'}$ is a cTFP along the first coordinate, then $a_i \neq a_j$ for all $i \neq j$. Similarly, $\mathcal{A}_{S'}$ being a cTFP along the second coordinate implies that $b_i \neq b_j$ for all $i \neq j$.

We can interpret these conditions in terms of the graph $\Gcal_S$. If none of the $a_i$ are equal to each other, then in each vertex in the first partite set is only connected to one other vertex in $\Gcal_S$. The analogous condition holds if none of the $b_i$ are equal to each other. 
Therefore the graph is a disjoint union of star graphs such that the internal vertices belong to the same partite set.

For the converse, let us assume that the $\mathcal{G}_S$ is a union of star graphs, such that each internal vertex corresponds to the first variable. Then, we will show that $\mathcal{A}_{S'}$ is a cTFP along the second coordinate. By the structure of $\mathcal{G}_S$, there is only one pair in $S$ that has $b$ as the second coordinate. Hence, by the conditions given in Theorem \ref{mainthm}, $\mathcal{A}_{S'}$ is a cTFP along the second coordinate. Similarly, if each internal vertex in the stars correspond to the second variable, then $\mathcal{A}_{S'}$ is a cTFP along the first coordinate. Therefore, if $\mathcal{G}_S$ is a union of star graphs, with each internal vertex corresponding to the same variable, then $\mathcal{A}_{S'}$ is a cTFP.

Hence, $\mathcal{A}_{S'}$ is a cTFP if and only if $\mathcal{G}_S$ is a disjoint union of stars, with internal vertices corresponding to the same variable. 
\end{proof}

Now, we recall a result of Brysiewicz and Maraj \cite{brysiewicz2023lawrence} that states that the ML-degree of the Lawrence lift of the model associated to $\mathcal{G}_S$ is equal to the number of spanning trees over the graph. In particular, the following holds. 

\begin{theorem}\cite{brysiewicz2023lawrence}
    When $\mathcal{G}_S$ is a forest, associated to a quasi-independence model, the ML-degree of the model associated to the Lawrence lift of $S$ is 1.
\end{theorem}

Example \ref{lleg} provides a model that is ML-degree $1$ but the associated matrix is not a cTFP. We can construct infinitely many examples of such models in this manner. Indeed, consider a model that is generated from the Lawrence lift of $\mathcal{A}_S$ where $\mathcal{G}_S$ is a forest. If $\mathcal{G}_S$ is not a forest of stars, then the model corresponding to the Lawrence lift has ML-degree $1$ but $\Lambda'(\mathcal{A}_S)$ is not a cTFP. However, it is possible that, as in the case of $2$-way quasi-independence models, there is another matrix $L$ with $\mathrm{rowspan}(L) = \mathrm{rowspan}(\Lambda'(\Gcal_S))$ such that $L$ is a cTFP. We leave this as an open question.

\section{ML-Degrees and Facial Submodels}
In this section, we present a sufficient condition for a quasi-independence model to have ML-degree greater than one and a counterexample to the converse condition. We first review background on polytopes associated to log-linear models. 

Let $\mathcal{A} \in \mathbb{Z}^{m\times n}$ be the defining matrix of a log-linear model $\mathcal{M}_\Acal$. We can construct a polytope $\mathcal{P}_\Acal$ that is associated to the model by defining it as the convex hull of the columns of $\mathcal{A}$. We denote this convex hull as conv$(\mathcal{A})$.

\begin{definition}
    Let $\mathcal{A}'$ be a matrix whose columns are a subset of the columns of $\mathcal{A}$. Then, if $\mathcal{P}_{\mathcal{A}'}$ is a face of the polytope $\mathcal{P}_\mathcal{A}$, we say that $\mathcal{A}'$ is a \textbf{facial submatrix} of $\mathcal{A}$. In terms of the corresponding statistical models, we say that $\mathcal{M}_{\mathcal{A}'}$ is a \textbf{facial submodel} of $\mathcal{M}_{\mathcal{A}}$.
\end{definition}

We use the following result of Coons and Sullivant \cite{coons2021quasi}.

\begin{theorem}\label{submodels}
    \cite{coons2021quasi} Let $\mathcal{A} \in \mathbb{Z}^{m \times n}$ and $\mathcal{A}' \in \mathbb{Z}^{d \times n}$ be the matrix that is formed of $d$ of the columns of $\mathcal{A}$. Suppose that $\mathcal{A}'$ is a facial submatrix of $\mathcal{A}$. Then if the log-linear model $\mathcal{M}_{\mathcal{A}}$ has ML-degree $1$, then $\Mcal_{\Acal'}$ also has ML-degree $1$.
\end{theorem}

Note that the contrapositive of the above theorem tells us that if a facial submodel of a model has ML-degree greater than 1, then the original model must also have ML-degree greater than 1. We use this result to prove Theorem \ref{slices}.

      Let $S \subset [n_1] \times \dots \times [n_k]$ specify the $k$-way quasi-independence model $\Mcal_S$ with $\Acal$-matrix $\Acal_S$. Let $a, b \in [k]$ be distinct indices and denote by $N_{ab}$ the product of each $[n_j]$ for $j \neq a,b$. Let $\mathbf{i} \in N_{ab}$.
    For each $(s, t) \in [n_a] \times [n_b]$, we define $\mathbf{i} + (s,t)$ to be the vector in $[n_1] \times \dots \times [n_k]$ whose $j$th coordinate is $i_j$ if $j \neq a,b$; $s$ if $j = a$; and $t$ if $j = b$. For fixed $a,b$ and $\mathbf{i}$, we define the set of ordered pairs
    \[
    S_{a,b}^{\mathbf{i}} = \{(s,t) \in [n_a] \times [n_b] \mid \mathbf{i} + (s,t) \in S \}.
    \]

\begin{theorem}\label{slices}
      Let $S \subset [n_1] \times \dots \times [n_k]$ with corresponding $\Acal$-matrix $\Acal_S$. If there exist distinct $a,b \in [k]$ and $\mathbf{i} \in N_{ab}$ such that the 2-way quasi-independence model specified by $S_{ab}^\mathbf{i}$ has ML-degree greater than 1, then $\Mcal_S$ also has ML-degree greater than 1.
\end{theorem}
\begin{proof}
    For each $\mathbf{i} \in N_{ab}$, we define $\Acal^{\mathbf{i}}$ to be the $\Acal$-matrix associated to the 2-way quasi-independence model specified by $S_{ab}^\mathbf{i}$. The order of the rows of the $\Acal$-matrix does not affect the associated log-linear model. So without loss of generality, we may consider the case where $a =1$ and $b=2$.  We define a column vector $\mathbf{y}^{\mathbf{i}}$ in $\{0,1\}^{[n_3]\times \dots \times [n_k]}$ with coordinates indexed by $(\ell, j)$ where $3 \leq \ell \leq k$ and $j \in [n_\ell]$. For each $\ell$, it has entries
    \[
    y^\mathbf{i}_{(\ell,j)} = \begin{cases}
        1, & \text{ if } j = i_\ell, \\
        0, & \text{ otherwise.}
    \end{cases}
    \]
We index the columns of $\Acal_S$ by elements of $S$ and group them according to the last $k-2$ coordinates. With this order on the columns, $\Acal_S$ is of the form
$$\mathcal{A}_S={ \left( \arraycolsep=1.4pt\def\arraystretch{2.2}
\begin{array}{c|c|c|c}
    \mathcal{A}^{1,\dots,1} & \cdots & \mathcal{A}^\mathbf{i} & \cdots \\
    \hline
    \mathbf{y}^{1,\dots,1} \dots \mathbf{y}^{1,\dots,1} & \cdots &\mathbf{y}^{\mathbf{i}}\dots \mathbf{y}^{\mathbf{i}} & \cdots
    \end{array} \right)}.$$

    Let $\mathcal{P}_S$ be the convex hull of the columns of $\Acal_S$.
    Then, consider the columns of the section of $\Acal_S$ of the~form

    $$\left( \arraycolsep=1.4pt\def\arraystretch{2.2}
\begin{array}{c}
     \mathcal{A}^{\mathbf{i}}\\
    \hline
    \mathbf{y}^{\mathbf{i}}\dots \mathbf{y}^{\mathbf{i}}
    \end{array} \right).$$
    
     Let $\mathcal{P}^\mathbf{i}$ denote the convex hull of the columns of this matrix. 
    We claim that $\mathcal{P}^\mathbf{i}$ is a face of $\mathcal{P}_S$. Indeed, we may obtain a $(n_1 + \dots +n_k)$-dimensional row vector $\mathbf{a}^\mathbf{i}$ by appending $n_1 + n_2$ zeros to the beginning of $(\mathbf{y}^\mathbf{i})^T$. Then $\mathbf{a}^{\mathbf{i}} \cdot \mathbf{x}$ is equal to $k-2$ for each vertex of $\mathcal{P}^{\mathbf{i}}$ and is strictly less than $k-2$ for every other column of $\Acal_S$. So $\mathcal{P}^{\mathbf{i}}$ is the face of $\mathcal{P}_S$ that maximizes the linear functional $\mathbf{a}^{\mathbf{i}}$, as needed.

    Note that all rows of this matrix below $\Acal^\mathbf{i}$ are either the row of all ones or the row of all zeros. So the quasi-independence model specified by this matrix is the same as that of $S^\mathbf{i}_{12}$. Hence for each $\mathbf{i}$, we have that $\Mcal_{\Acal^{\mathbf{i}}}$ is a facial submodel of $\Mcal_S$. The contrapositive of Theorem \ref{submodels} tells us that if there exists even one $\mathcal{A}^i$ such that $\mathcal{M}_{\mathcal{A}^i}$ has ML-degree greater than $1$, then the ML-degree of $\mathcal{M}_{S}$ must also be greater than~$1$.
\end{proof}

\subsection{No 3-way interaction model}
To construct a counterexample to the converse of Theorem \ref{slices}, consider the \emph{no $3$-way interaction model}. 
A no three-way interaction model is a model on three variables that parameterizes the distributions where all possible interactions involve only two of the three variables. In the most simple case, the three variables $X_1, X_2,$ and $X_3$ can each take values $1$ or $2$. Let  $S = \{(1,1,1), (1,1,2), (1,2,1), (1,2,2), (2,1,1), (2,1,2), (2,2,1), (2,2,2)\}$ so that the no $3$-way interaction model is $\Mcal_S$.

We now wish to verify two things: that the ML-degree of this model is greater than one and that the facial submodels of this model are ML-degree one. 
 It is well known that that the ML-degree of this model is $3$, see for example \cite{amendola2019maximum,brysiewicz2023lawrence}. 
 A simple Macaulay2 computation shows the polytope associated to $\mathcal{M}_S$, is simplicial (i.e., each of the facets of this polytope are simplices). Therefore, the ML-degree of the facial submodels associated to the facets are $1$. Finally, as each face of the polytope is contained in a facet, each of the facial submodels of $\mathcal{M}_S$ are contained within a simplicial model, and hence are each of ML-degree $1$. 

Thus, this is a model that is comprised of facial submodels with rational MLE although the entire model has ML-degree greater than one. Hence, this serves as a counterexample to the converse of Theorem \ref{slices}. 

\subsection*{Acknowledgements}
J.I.C. is grateful for support from the L’Oréal-UNESCO
For Women in Science UK and Ireland Rising Talent Award Programme.
H.A.H. gratefully acknowledges funding from the
Royal Society RGF$\backslash$EA$\backslash$201074 and UF150238 and the UK Centre for Topological Data Analysis EPSRC grant EP/R018472/1. We thank the Royal Society Enhancement grant RGF$\backslash$EA$\backslash$201074 that supported N.C.P.'s undergraduate research experience.

\bibliographystyle{alpha}
\bibliography{ref}

\newcommand{\etalchar}[1]{$^{#1}$}
\begin{thebibliography}{BLVS{\etalchar{+}}99}

\bibitem[ABB{\etalchar{+}}19]{amendola2019maximum}
Carlos Am{\'e}ndola, Nathan Bliss, Isaac Burke, Courtney~R Gibbons, Martin Helmer, Serkan Ho{\c{s}}ten, Evan~D Nash, Jose~Israel Rodriguez, and Daniel Smolkin.
\newblock The maximum likelihood degree of toric varieties.
\newblock {\em Journal of Symbolic Computation}, 92:222--242, 2019.

\bibitem[AKK20]{amendola2020maximum}
Carlos Am{\'e}ndola, Dimitra Kosta, and Kaie Kubjas.
\newblock Maximum likelihood estimation of toric fano varieties.
\newblock {\em Algebraic Statistics}, 11(1):5--30, 2020.

\bibitem[BLVS{\etalchar{+}}99]{orientedmatroids}
Anders Björner, Michel Las~Vergnas, Bernd Sturmfels, Neil White, and Gunter~M. Ziegler.
\newblock {\em Oriented Matroids}.
\newblock Encyclopedia of Mathematics and its Applications. Cambridge University Press, 2 edition, 1999.

\bibitem[BM23]{brysiewicz2023lawrence}
Taylor Brysiewicz and Aida Maraj.
\newblock Lawrence lifts, matroids, and maximum likelihood degrees.
\newblock {\em arXiv preprint arXiv:2310.13064}, 2023.

\bibitem[BR17]{bocci2017exact}
Cristiano Bocci and Fabio Rapallo.
\newblock Exact tests to compare contingency tables under quasi-independence and quasi-symmetry.
\newblock {\em Journal of Algebraic Statistics}, 13(1), 2017.

\bibitem[BS17]{bernstein2017unimodular}
Daniel~Irving Bernstein and Seth Sullivant.
\newblock Unimodular binary hierarchical models.
\newblock {\em Journal of Combinatorial Theory, Series B}, 123:97--125, 2017.

\bibitem[CHKS06]{catanese2006maximum}
Fabrizio Catanese, Serkan Ho{\c{s}}ten, Amit Khetan, and Bernd Sturmfels.
\newblock The maximum likelihood degree.
\newblock {\em American Journal of Mathematics}, 128(3):671--697, 2006.

\bibitem[CI88]{colombo1988quasi}
AG~Colombo and P~Ihm.
\newblock A quasi-independence model to estimate failure rates.
\newblock {\em Reliability Engineering \& System Safety}, 21(4):309--318, 1988.

\bibitem[CLR23]{coons2022rational}
Jane~Ivy Coons, Carlotta Langer, and Michael Ruddy.
\newblock Classical iterative proportional scaling of log-linear models with rational maximum likelihood estimator.
\newblock {\em International Journal of Approximate Reasoning}, 165:109043, 2023.

\bibitem[CS21]{coons2021quasi}
Jane~Ivy Coons and Seth Sullivant.
\newblock Quasi-independence models with rational maximum likelihood estimator.
\newblock {\em Journal of Symbolic Computation}, 104:917--941, 2021.

\bibitem[DMS21]{duarte2021discrete}
Eliana Duarte, Orlando Marigliano, and Bernd Sturmfels.
\newblock Discrete statistical models with rational maximum likelihood estimator.
\newblock {\em Bernoulli}, 27(1), 2021.

\bibitem[Goo94]{goodman1994quasi}
Leo~A Goodman.
\newblock On quasi-independence and quasi-dependence in contingency tables, with special reference to ordinal triangular contingency tables.
\newblock {\em Journal of the American Statistical Association}, 89(427), 1994.

\bibitem[HS14]{huh2014likelihood}
June Huh and Bernd Sturmfels.
\newblock Likelihood geometry.
\newblock {\em Combinatorial algebraic geometry}, 2108(63):1305--7462, 2014.

\bibitem[Huh13]{huh2013maximum}
June Huh.
\newblock The maximum likelihood degree of a very affine variety.
\newblock {\em Compositio Mathematica}, 149(8):1245--1266, 2013.

\bibitem[Lau96]{lauritzen1996graphical}
Steffen~L Lauritzen.
\newblock {\em Graphical models}, volume~17.
\newblock Clarendon Press, 1996.

\bibitem[RS16]{RAUH2016276}
Johannes Rauh and Seth Sullivant.
\newblock Lifting markov bases and higher codimension toric fiber products.
\newblock {\em Journal of Symbolic Computation}, 74:276--307, 2016.

\bibitem[Stu96]{sturmfels1996grobner}
Bernd Sturmfels.
\newblock {\em Grobner bases and convex polytopes}, volume~8.
\newblock American Mathematical Soc., 1996.

\bibitem[Sul07]{sullivant2007toric}
Seth Sullivant.
\newblock Toric fiber products.
\newblock {\em Journal of Algebra}, 316(2):560--577, 2007.

\bibitem[Sul18]{sullivant2018algebraic}
Seth Sullivant.
\newblock {\em Algebraic statistics}, volume 194.
\newblock American Mathematical Soc., 2018.

\end{thebibliography}

\end{document}